\documentclass[a4paper, reqno]{amsart}

\usepackage[all]{xy}
\usepackage[english]{babel}
\usepackage[utf8]{inputenc}
\usepackage{amssymb,amsmath,amsthm}
\usepackage{amsfonts}
\usepackage{graphicx}
\usepackage{psfrag}
\usepackage{url}
\usepackage[dvipsnames]{xcolor}
\usepackage[hidelinks]{hyperref}
\usepackage{csquotes}
\usepackage[alphabetic, msc-links, initials]{amsrefs}
\usepackage{tikz-cd}
\usepackage{esint}
\theoremstyle{plain}
\newtheorem{thm}{Theorem}[section]
\newtheorem{prop}[thm]{Proposition}
\newtheorem{cor}[thm]{Corollary}
\newtheorem{lem}[thm]{Lemma}

\theoremstyle{definition}
\newtheorem{defn}[thm]{Definition}

\newtheorem{assm}[thm]{Assumption}

\newtheorem{exmp}[thm]{Example}

\theoremstyle{remark}
\newtheorem{rem}[thm]{Remark}

\theoremstyle{plain}

\newcommand{\R}{\mathbb{R}}

\newcommand{\C}{\mathbb{C}}

\newcommand{\N}{\mathbb{N}}

\newcommand{\grad}{\operatorname{grad}}
\newcommand{\supp}{\operatorname{supp}}

\newcommand{\scal}{\operatorname{scal}}
\newcommand{\ric}{\operatorname{Ric}}
\newcommand{\trace}{\operatorname{tr}}
\newcommand{\kernel}{\operatorname{ker}}

\newcommand{\volume}{\operatorname{vol}}

\newcommand{\dv}{\text{ }dV}

\newcommand{\Diff}{\operatorname{Diff}}

\newcommand{\spectrum}{\operatorname{spec}}
\newcommand{\gradient}{\operatorname{grad}}

\newcounter{mnotecount}[section]

\allowdisplaybreaks

\title[Local and global scalar curvature rigidity of Einstein manifolds]{Local and global scalar curvature rigidity of Einstein manifolds}

\author{Mattias Dahl}
\author{Klaus Kröncke}

\email{dahl@kth.se}
\email{kroncke@kth.se}
\address{Institutionen för Matematik, Kungliga Tekniska Högskolan, 100 44 Stockholm, Sweden}

\begin{document}
\hbadness=100000
\vbadness=100000

\begin{abstract}
An Einstein manifold is called scalar curvature rigid if there are no compactly supported volume-preserving deformations of the metric which increase the scalar curvature. We give various characterizations of scalar curvature rigidity for open  Einstein manifolds as well as for closed Einstein manifolds. As an application, we construct mass-decreasing deformations of the Riemannian Schwarzschild metric and the Taub-Bolt metric.
\end{abstract}

\maketitle
\tableofcontents

\begin{sloppypar}

\section{Introduction and main results}

We call a Riemannian metric $\hat{g}$ on a manifold $M$ \emph{scalar curvature rigid} (or SCR for short) if it can not be deformed to a metric $g$ such that $g$ agrees with $\hat{g}$ outside a compact set and such that $\mathrm{scal}^g\geq\scal^{\hat{g}}$ everywhere and $\scal^g>\scal^{\hat{g}}$ somewhere.

Euclidean space is well known to be SCR due to the rigidity part of the positive mass theorem for asymptotically Euclidean manifolds \cite{SchoenYauPMT,witten1981new}.
Similarly, a version of the positive mass theorem for asymptotically hyperbolic manifolds implies that hyperbolic space \cite{min1989scalar,chrusciel2003mass} is SCR. Inspired by these examples, it was expected that the upper half of the round sphere would be SCR as well but this conjecture was surprisingly disproved by Brendle-Marques-Neves \cite{BMN2011}.

Concepts of mass exist also for other asymptotics such as asymptotically locally Euclidean (ALE) and asymptotically (locally) flat (AF/ALF) and it is natural to ask about consequences for scalar curvature rigidity in these cases. In fact, the assertion of the positive mass theorem does not hold for ALE manifolds in general \cite{LB88} but there is a version for ALE spin manifolds by the first author which implies that ALE manifolds with parallel spinors \cite{dahl1997} are SCR. A version of the positive mass theorem also exists for AF/ALF metrics which implies that $\R^{n-1}\times S^1$ is SCR \cite{MinerbeMassALF}.

The purpose of this paper is to systematically characterize scalar curvature rigidity of Einstein manifolds. It is natural to restrict to the class of Einstein metrics as the scalar curvature of a non-Einstein metric can (at least in the compact case) be increased, for example by evolving it along the (normalized) Ricci flow. If the scalar curvature of the Einstein metric is nonzero, we will impose a volume constraint on the deformations of the metric.

We will state and prove equivalent characterizations of scalar curvature rigidity on closed as well as on open Einstein manifolds. In particular, our manifolds do not need to be complete and our results also allow to detect which subsets of compact or complete noncompact Einstein manifolds are SCR. Scalar curvature rigidity is characterized by means of positivity of the \emph{Einstein operator} which is an elliptic operator closely related to the well-known \emph{Lichnerowicz Laplacian}.

Our results imply that Ricci-flat manifolds with parallel spinors are SCR. This fits well with Witten's proof of the positive mass theorem which does also work in the ALE case and as mentioned before, implies that Ricci-flat ALE manifolds with parallel spinors are SCR. In contrast, the Riemannian Schwarzschild metric and the Taub-Bolt metric are Ricci-flat ALF metrics which do not have parallel spinors and we prove in this paper that both metrics are not SCR. This allows us additionally to show that a positive mass theorem does not hold for these metrics: There are small scalar-flat perturbations of these metrics which decrease the mass.

\subsection{Closed Einstein manifolds}

In the following we give equivalent characterizations of scalar curvature rigidity of closed Einstein manifolds. 
In the Ricci-flat case, many implications between the different conditions (stated in Theorem \ref{thm:rigiditycptRicciflat} below) were shown in the literature. However, we could not find a formulation in this complete form. Before stating the theorem, let us recall that the (conformal) Yamabe invariant of a conformal class of metrics is given by
\[
Y(M,[g]):=\inf_{\tilde{g}\in [g]} \mathrm{vol}(M,\tilde{g})^{\frac{2}{n}-1}\int_M\scal^{\tilde{g}}\dv^{\tilde{g}}.
\]
In the following we think of the conformal Yamabe invariant as the functional $Y:\mathcal{M}\ni g\mapsto Y(M,[g])$,
where $\mathcal{M}$ is the space of smooth Riemannian metrics on the manifold $M$. The topology we use for the locality statements in the remainder of the section is the $C^{2,\alpha}$-topology on the space of metrics.
\begin{thm}\label{thm:rigiditycptRicciflat}
Let $(M,\hat{g})$ be a closed Ricci-flat manifold. Then the following are equivalent.
\begin{itemize}
\item[(i)] $\hat{g}$ is a local maximizer of the conformal Yamabe invariant.
\item[(ii)] Close to $\hat{g}$, there is no metric $g$ with $\scal^{g}\geq0$ and $\scal^{g}>0$ somewhere.
\item[(iii)] Close to $\hat{g}$, there is no metric of constant positive scalar curvature.
\item[(iv)] Any scalar-flat metric close to $\hat{g}$ is also Ricci-flat.
\item[(v)] $\hat{g}$ is dynamically stable under the Ricci flow.
\end{itemize}
\end{thm}
The equivalence (i)$\Leftrightarrow$(ii) follows from the solution of the Yamabe problem \cite{schoen1984conformal}. The implication (ii)$\Rightarrow$(iii) is trivial and the converse implication (iii)$\Rightarrow$(ii) follows from the structure of the space of constant scalar curvature metrics, see \cite[Theorem 2.5]{koiso1979decomposition} and again \cite{schoen1984conformal}. A central ingredient from the Yamabe problem used here is the fact that the sign of the Yamabe invariant of a conformal class is the same as the sign of any constant scalar curvature metric in it.

The implication (ii)$\Rightarrow$(iv) is widely attributed to Bourguignon and is carried out in detail in \cite[Proposition 2.1]{DaiWangWei05}. The converse implication (iv)$\Rightarrow$(ii) is less straightforward and requires some more work. We will carry out the arguments for the Einstein case stated below.
Finally, the equivalence (i)$\Leftrightarrow$(v) follows from using Perelman's $\lambda$-functional, and more precisely, the assertion
\begin{itemize}
\item[(vi)] $\hat{g}$ is a local maximizer of the $\lambda$-functional.
\end{itemize}
In fact, (v)$\Rightarrow$(vi) follows from the monotonicity of $\lambda$ along the Ricci-flow \cite{perelman2002entropy} and (vi)$\Rightarrow$(v) 
is Theorem 1 in \cite{haslhofer2014dynamical}. The remaining equivalence (i)$\Leftrightarrow$(vi) is Theorem 1.1 in \cite{Kroenckearxiv2013} by the second author. 

A criterion for all these conditions to hold can be formulated in terms of the Einstein operator which we introduce now. Recall that a symmetric two-tensor is called a transverse traceless tensor, or TT-tensor, if its trace and its divergence both vanish identically. The space of TT-tensors is denoted by $TT$, if we do not specify any regularity.
\begin{defn}
The \emph{Einstein operator} $\Delta_E:C^{\infty}(S^2M)\to C^{\infty}(S^2M)$ is defined by $\Delta_E=\nabla^*\nabla-2\mathring{R}$, where $\mathring{R}h_{ij}=h^{kl}R_{iklj}$, where we use the Einstein summation convention. A closed Einstein manifold is called \emph{linearly stable} if all eigenvalues of $\Delta_E|_{TT}$ are nonnegative and \emph{linearly unstable} otherwise. We call an Einstein manifold \emph{integrable}, if all $h\in \kernel(\Delta_E|_{TT})$ are tangent to smooth families of Einstein metrics.
\end{defn}
\begin{rem}
The equivalent conditions in Theorem \ref{thm:rigiditycptRicciflat} do hold if $(M,\hat{g})$ is linearly stable and integrable \cite{Ses06}. Linear stability and integrability are satisfied by closed Ricci-flat manifolds whose universal cover admits a parallel spinor \cite{Wang91,DaiWangWei05,ammann2019holonomy}. This is a large class of closed Ricci-flat manifolds, which contains all known examples of such manifolds. 
It is a major open problem whether examples outside this class exist.
\end{rem}
It is natural to ask for an analogue of Theorem \ref{thm:rigiditycptRicciflat} for closed Einstein manifolds. In this case, it is essential to impose a volume constraint in order to avoid rescalings of metrics which obviously just rescale the scalar curvature. To formulate the theorem, let $\mathcal{M}_c$ be the set of smooth Riemannian metrics on $M$ of volume $c>0$.
\begin{thm}\label{thm:rigiditycptEinstein}
Let $(M,\hat{g})$ be a closed Einstein manifold and $c=\mathrm{vol}(M,\hat{g})$. Then the following are equivalent.
\begin{itemize}
\item[(i)] $\hat{g}$ is a local maximizer of the Yamabe invariant.
\item[(ii)] Close to
$\hat{g}$, there is no metric $g\in\mathcal{M}_c$ such that $\scal^{g}\geq \scal^{\hat{g}}$ and $\scal^{g}>\scal^{\hat{g}}$ somewhere.
\item[(iii)] Close to $\hat{g}$, there is no constant scalar curvature metric $g\in\mathcal{M}_c$ such that $\scal^{g}> \scal^{\hat{g}}$.
\item[(iv)] Any metric $g\in\mathcal{M}_c$ close to $\hat{g}$ with $\scal^{g}=\scal^{\hat{g}}$ is also Einstein.
\end{itemize}
If $\scal^{\hat{g}}\leq0$, then (i)-(iv) are also equivalent to
\begin{itemize}
\item[(v)] $\hat{g}$ is dynamically stable under the volume-normalized Ricci flow.
\end{itemize}
\end{thm}

The equivalence (i)$\Leftrightarrow$ (v) is \cite[Theorem 1.2]{Kroenckearxiv2013} by the second author.
We will prove the equivalence of (i)-(iv) in Section \ref{sec:global_rigidity}.
In contrast to Theorem \ref{thm:rigiditycptRicciflat}, the results on the Yamabe problem can not be used to show equivalence of (i),(ii) and (iii) as the sign of the Yamabe invariant does not play any role here. Two ingredients are essential: On the one hand, we use a parameter-dependent generalization of the $\lambda$-functional. On the other hand, we provide a new structure theorem of the space of metrics close to $\hat{g}$ which extends both Koiso's structure theorem  \cite{koiso1979decomposition} and Ebin's slice theorem \cite{ebin1970manifold}.
\begin{rem}
In Theorem \ref{thm:rigiditycptEinstein}, $(v)\Rightarrow (i)$ does hold when $\scal^g>0$ as well, but the converse then false. The complex projective space is a prominent counterexample \cite[Corollary 1.8]{Kroenckearxiv2013}. However, it seems reasonable to believe that if (i) holds, $\hat{g}$ is dynamically stable with respect to an adapted version of the Ricci flow, for example the Ricci-Burguignon flow \cite{catino2017ricci} or the conformal Ricci flow \cite{Fis04}.
\end{rem}
\begin{rem}
As in the Ricci-flat case, the conditions (i)-(iv) do hold if $(M,\hat{g})$ is linearly stable and integrable. If  $\scal^{\hat{g}}<0$, all known examples satisfy both conditions \cite{Dai07}. In contrast, there are many known unstable closed Einstein manifolds with $\scal^{\hat{g}}>0$, see \cite{CH15} for examples.
\end{rem}

\subsection{Open Einstein manifolds}

For the characterization of scalar curvature rigidity of open Einstein manifolds, we need an appropriate definition of linear stability in this setting. Let $C^{\infty}_c(TT)$ be the space of compactly supported TT-tensors on the manifold $M$.
\begin{defn} \label{def-stable-open}
An open Einstein manifold $(M,g)$ is called \emph{linearly stable}, if the number
\[
\mu_1(\Delta_E|_{TT},M):=\inf\left\{(\Delta_Eh,h)_{L^2}\mid h\in C^{\infty}_c(TT),\left\|h\right\|^2_{L^2}=1\right\}
\]
is nonnegative and \emph{linearly unstable} otherwise.
\end{defn}
The number $\mu_1(\Delta_E|_{TT},M)$ is the bottom of the $L^2$-spectrum of
of the Einstein operator acting on TT-tensors. If $M$ is the interior of a compact manifold $\overline{M}$ with smooth boundary, $\mu_1(\Delta_E|_{TT},M)$ coincides with the smallest Dirichlet eigenvalue of $\Delta_E|_{TT}$ on $\overline{M}$. This follows from the nontrivial fact that $C^{\infty}_c(TT)$ is $H^1$-dense in $\mathring{H}^1(S^2M)\cap TT$. The proof of this fact follows from results of Delay \cite{Delay2012}, and is carried out in detail in Section 3.3.

The theorem in the Ricci-flat case can be now formulated as follows:
\begin{thm}\label{thm:rigidityopenRicciflat}
Let $(M,\hat{g})$ be an open Ricci-flat manifold which does not admit a linear function, that is, there is no nonconstant function $f$ with $\nabla^2f\equiv0$. Then the following are equivalent:
\begin{itemize}
\item[(i)] $(M,\hat{g})$ is linearly stable.
\item[(ii)] Close to $\hat{g}$, there is no metric $g$ with $g-\hat{g}|_{M\setminus K}\equiv0$ for some compact set $K\subset M$ which additionally satisfies
\[
\scal^g\geq 0, \quad \scal^{\hat{g}}(p)>0 \text{ for some }p\in M.
\]
\item[(iii)] If $g$ is a metric close to $\hat{g}$ with $\scal^{g}\equiv0$ and $g-\hat{g}|_{M\setminus K}\equiv0$ for some compact set $K\subset M$, then $g$ is isometric to $\hat{g}$.
\end{itemize}
\end{thm}
Conditions (i),(ii) and (iii) can be seen as respective replacements of the conditions (i), (ii) and (iv) in Theorem \ref{thm:rigiditycptRicciflat}. Conditions (iii) and (v) in Theorem \ref{thm:rigiditycptRicciflat} do not have appropriate replacements in the present situation. We cannot construct perturbations with larger constant scalar curvature if we just allow compactly supported perturbations of $\hat{g}$. Finally, dynamical stability does not make sense in this context either as there is no canonical Ricci flow on an incomplete manifold.

The proof of theorem \ref{thm:rigidityopenRicciflat} will not be carried out in detail as it follows, up to minor modifications, along the lines of the proof of the corresponding theorem in the general Einstein case.
There, we have to additionally assume a spectral inequality and to
impose a volume constraint on the support of the perturbations. The theorem reads as follows:
\begin{thm}\label{thm:rigidityopenEinstein}
Let $(M,\hat{g})$ be an open Einstein manifold satisfying the following two assumptions:
\begin{itemize}
\item[(A)] $(M,\hat{g})$ is not locally isometric to a warped product.
\item[(B)] If $\scal^{\hat{g}}>0$, $M$ is the interior of a compact manifold $\overline{M}$ with smooth boundary whose first nonzero Neumann eigenvalue satisfies
\begin{equation} \label{eq_spectral_inequality}
\mu_1^{NM}(\Delta^{\hat{g}},\overline{M})>\frac{\scal^{\hat{g}}}{n-1}.
\end{equation}
\end{itemize}
Then the following are equivalent:
\begin{itemize}
\item[(i)] $(M,\hat{g})$ is linearly stable.
\item[(ii)] Close to $\hat{g}$, there is no metric $g$ with 
\[
g-\hat{g}|_{M\setminus K}\equiv0,\qquad \volume(K,g)=\volume(K,\hat{g})
\]
for some compact set $K\subset M$ which additonally satisfies
\[
\scal^g\geq\scal^{\hat{g}}, \quad 
\scal^g(p)>\scal^{\hat{g}}(p) \text{ for some }p\in M.
\]
\item[(iii)] If $g$ is a metric close to $\hat{g}$ with $\scal^{g}\equiv\scal^{\hat{g}}$ and
\[
g-\hat{g}|_{M\setminus K}\equiv0, \quad
\volume(K,g)=\volume(K,\hat{g})
\]
for some compact set $K\subset M$, then $g$ is isometric to $\hat{g}$.
\end{itemize}
\end{thm}

\begin{rem}
The two implications (i)$\Rightarrow$(ii) and (i)$\Rightarrow$(iii) do also hold without Assumption (A), as it does not appear in the proof. 
Conversely, the two implications (ii)$\Rightarrow$(i) and (iii)$\Rightarrow$(i) do also hold without Assumption (B).
\end{rem}
\begin{rem}
Regarding the assumptions in Theorem \ref{thm:rigidityopenEinstein}:
\begin{itemize}
\item[(i)] If $(M,\hat{g})$ is an open subset of a closed Einstein manifold $(N,\hat{g})$, Assumption (A) is satisfied unless $(M,\hat{g})$ is of constant nonzero curvature or $(N,\hat{g})$ is a Ricci-flat product manifold with a flat factor.
\item[(ii)] Assumption (B) holds whenever $\overline{M}$ is a compact manifold with a convex boundary, see \cite[Theorem 4.3]{Esc90}.
\end{itemize}
\end{rem}

The implications (i)$\Rightarrow$(ii) and (i)$\Rightarrow$(iii) are proven together. As in the closed case, we use a parameter-dependent generalization of Perelman's $\lambda$-functional which we call $\lambda_{\alpha}$. In the closed case, $\lambda_{\alpha}$ assigns to every metric the smallest eigenvalue of an elliptic operator. To recover a similar variational structure in the presence of a nonempty boundary it turns out to be convenient to impose Neumann boundary conditions for this functional. Assumption (B) will give us the right sign for the second variation of $\lambda_{\alpha}$. A detailed discussion of $\lambda_{\alpha}$ is carried out in Section \ref{sec_lambda} and Section \ref{sec_local_rigidity}, culmulating in Theorem \ref{thm:local_maximality_lambda_boundary}.

The other part of the proof consists in proving the implications $\neg$(i)$\Rightarrow\neg$(ii) and $\neg$(i)$\Rightarrow\neg$(iii). Assuming that $(M,\hat{g})$ is linearly unstable, we will construct volume-preserving perturbations of the metric $\hat{g}$ with larger scalar curvature respectively with the same scalar curvature but which are not Einstein. Linear instability enters in an essential way since the crucial term in $D^2_g\scal(h,h)$ is given by $-1/2\langle \Delta_Eh,h\rangle$. The construction of these perturbations follows from a carefully executed second order implicit function argument using the solution theory of underdetermined elliptic equations on manifolds as developed by Delay \cite{Delay2012}. Assumption (A) guarantees solvability of the equations that appear. The details of this construction is carried out in Section \ref{sec_prescribe_scal}. All the arguments outlined in these two paragraphs are brought together at the end of Section \ref{sec_local_rigidity}.

The typical example to which Theorem \ref{thm:rigidityopenEinstein} applies is an open subset $\Omega$ of a complete Einstein manifold $M$. To simplify the following discussion, let us assume in this remark that $\Omega$ does satisfy the Assumptions (A) and (B).

By domain monotonicity of $\mu_1(\Delta_E,\Omega)$ and the fact that $\mu_1(\Delta_E,\Omega)\to \infty$ as $\Omega$ shrinks to a point, we see that sufficiently small open subsets of a given Einstein manifold $M$ are SCR.

If $M$ is linearly unstable, open subsets $\Omega\subset M$ are SCR whenever they are so small that $\mu_1(\Delta_E,\Omega)> 0$. If $\Omega$ is so large that $\mu_1(\Delta_E,\Omega)=0$, it is still SCR because we restrict to perturbations with compact support in $\Omega$. Each such perturbation will then be supported in a slightly smaller domain $\Omega'\subset \Omega$ where we have $\mu_1(\Delta_E,\Omega')> 0$ by monotonicity. If $\Omega$ is taken only slightly larger, we will immediatly have $\mu_1(\Delta_E,\Omega)<0$ and scalar curvature rigidity fails.

It is open at the moment whether in the case $\mu_1(\Delta_E,\Omega)=0$, we also get scalar curvature rigidity with respect to perturbations whose supports do hit the boundary of $\Omega$. This question presumably depends on the integrability of elements $h\in \ker(\Delta_E|_{TT})$ with Dirichlet boundary conditions on $\partial\Omega$ and will be subject of further investigations.
\begin{exmp}
Any open non-product Ricci-flat manifold whose universal cover carries a parallel spinor, is linearly stable and therefore SCR by Theorem \ref{thm:rigidityopenRicciflat}.
This applies to all known examples of Ricci-flat manifolds which are either closed or ALE. Examples of Ricci-flat AF/ALF manifolds with parallel spinors are provided by $\R^{n-1}\times S^{1}$ and the Taub-NUT metric.

There are also Ricci-flat manifolds which are linearly stable but do not admit parallel spinors. In \cite{kroncke2017stable}, the second author showed that Ricci-flat cones over $S^n\times S^m$ are linearly stable if $n+m\geq 9$.
\end{exmp}

\begin{exmp}
In other geometric situations than closed or ALE, there are many examples of unstable Ricci-flat manifolds which by Theorem \ref{thm:rigidityopenRicciflat} are not SCR:  
\begin{itemize}
\item The Riemannian Schwarzschild metic and the Taub-Bolt metric are both linearly unstable AF/ALF Ricci-flat metrics.
\item A Ricci-flat cone over a product Einstein manifold is linearly unstable if $n<10$ \cite{hhs2014}.
\item B\"{o}hm constructed complete noncompact Ricci-flat manifolds which are asymptotic to such cones \cite{bohm1999non}. These examples are linearly unstable as well.
\end{itemize}
Similarly, in constrast to the closed case, there are many examples of noncompact linearly unstable Einstein manifolds of negative scalar curvature which by Theorem \ref{thm:rigidityopenEinstein} are not SCR:
\begin{itemize}
\item The AdS Riemannian Schwarzschild metric and the AdS-Taub-Bolt metric are both linearly unstable for certain parameters \cite{Pre00,Warnick}.
\item A hyperbolic cone over a product Einstein manifold is linearly unstable if $n<10$, see \cite[Theorem 4.7]{kroncke2017stable} by the second author.
\end{itemize}
\end{exmp}

\begin{rem}
There is an interesting analogy of this to geodesics and minimal surfaces which are minimizers of the respective area and energy functionals if they are small, until they may reach a cricital size at which the respective smallest Dirichlet eigenvalues change sign. In that sense, the boundary $\partial \Omega$ of a domain $\Omega\subset M$ with $\mu_1(\Delta_E,\Omega)=0$ is an analogue of a pair of conjugate points. The present problem is much more complicated however, as we do not only want to increase a single functional such as area or energy but a function on the domain.
\end{rem}
\subsection*{Acknowledgements}
This work was initiated at Institut Mittag-Leffler during the program \emph{General Relativity, Geometry and Analysis} in Fall 2019 and completed at Mathematisches Forschungsinstitut Oberwolfach during the conference  	\emph{Analysis, Geometry and Topology of Positive Scalar Curvature Metrics} in Summer 2021.
We wish to thank both institutes for their hospitality and for the excellent working conditions they provide. The work of the second author was supported by the Deutsche Forschungsgemeinschaft (KR
4978/1-1) through the priority program 2026 \emph{Geometry at Infinity}. We also want to thank the referee for a very careful reading and many helpful comments to improve the paper.

\section{Notation, conventions, and formulas}

Throughout the paper, any Riemannian metric will be smooth, unless stated otherwise.
The Riemann curvature tensor of a Riemannian metric $g$ is defined by the sign convention that $R_{ijkl}=g(\nabla_{\partial_i}\nabla_{\partial_j}\partial_k-\nabla_{\partial_j}\nabla_{\partial_i}\partial_k,\partial_l)$. The Ricci curvature and the scalar curvature of a metric $g$ are denoted by $\ric^g,\scal^g$, respectively. The Laplacian on functions is defined by $\Delta^gf=\nabla^*\nabla f=-g^{ij}\nabla^2_{ij}f$, where we use the Einstein summation convention. The volume element is denoted by $\dv^g$.

The bundle of symmetric two-tensors is denoted by $S^2M$ while the subbundle of tracefree symmetric two-tensors (with respect to a given metric) is denoted by $S^2_0M$. For the space of sections of a vector bundle $E$ with regularity for example $C^{\infty}$, $C^{k,\alpha}$, we write $C^{\infty}(E)$, $ C^{k,\alpha}(E)$ etc. Although TT-tensors do not form a bundle, we write for notational convenience $C^{\infty}(TT)$ or $C^{k,\alpha}(TT)$ for the TT-tensors with respective regularity.

The divergence of a symmetric two-tensor and of a one-form are defined by $\delta h_k=-g^{ij}\nabla_ih_{jk}$ and $\delta\omega=-g^{ij}\nabla_i\omega_j$, respectively. The formal adjoint $\delta^*:C^{\infty}(T^*M)\to C^{\infty}(S^2M)$ is given by $(\delta^*\omega)_{ij}=\frac{1}{2}(\nabla_i\omega_j+\nabla_j\omega_i)$. 
For $h,k\in C^{\infty}(S^2M)$, we define another two-tensor $h\circ k$ by $(h\circ k)_{ij}=h_{im}g^{ml}k_{lj}$. Furthermore, we define the Lichnerowicz Laplacian $\Delta_L: C^{\infty}(S^2M)\to  C^{\infty}(S^2M)$ by 
\[
 \Delta_Lh=\nabla^*\nabla h+\ric\circ h+h\circ\ric-2\mathring{R}h.
\]
Note that in the Einstein case $\Delta_L=\Delta_E+2\sigma$, if $\sigma$ is the Einstein constant.
 
The following well-known variation formula is essential for our main results.
For the proof see \cite[Theorem 1.174]{Besse}.
\begin{lem}\label{lem_first_variation_scal}
The first variation of the scalar curvature is given by
\[
D_g\scal(h)=\frac{d}{dt}\scal^{g+th}|_{t=0} =
\Delta\trace h + \delta(\delta h) - \langle\ric,h\rangle.
\]
In particular, if $g$ is Einstein with $\ric=\sigma g$, we have
\[
D_g\scal(h)=\frac{d}{dt}\scal^{g+th}|_{t=0} =
\Delta\trace h + \delta(\delta h) - \sigma\trace h.
\]
\end{lem}
For the purposes of this paper, we have to go beyond the first variation formula and compute the second variation formula of the scalar curvature.
\begin{lem}\label{lem_second_variation_scal}
The second variation of the scalar curvature is given by
\[ \begin{split}
D^2_g\scal(h,h)
&= \frac{d^2}{dt^2}\scal^{g+th} |_{t=0} 
= \langle h,\nabla^2\trace h\rangle-\langle \delta h+\frac{1}{2}\nabla \trace h,\nabla \trace h\rangle-\Delta |h|^2\\
&\qquad  +\langle h,\nabla\delta h\rangle-|\delta h|^2-\frac{1}{2}\langle\nabla\trace h,\delta h\rangle
+\delta(\delta'h)\\
&\qquad-\langle\frac{1}{2}\Delta_L h-\delta^*(\delta h)-\frac{1}{2}\nabla^2\trace h,h\rangle+2 \langle \ric,h \circ h\rangle.
\end{split} \]
where the prime $'$ is shorthand notation for the first variation in the direction of $h$.
In particular, if $g$ is Einstein such that $\ric=\sigma g$ and $h$ is a TT-tensor, we have
\[
D^2_g\scal(h,h)=
-\Delta(|h|^2)+\delta(\delta'h)
-\frac{1}{2}\langle\Delta_E h,h\rangle +\sigma|h|^2. 
\]
\end{lem}
\begin{rem}
The precise form of $\delta'h$ is irrelevant for our purposes which is why we do not write it out explicitly.
\end{rem}
\begin{proof}[Proof of Lemma \ref{lem_second_variation_scal}]
From Lemma \ref{lem_first_variation_scal} we have
\begin{align*}
D^2_g\scal(h,h)
&=\frac{d}{dt}(\Delta\trace h + \delta(\delta h) - \langle\ric^{g},h\rangle) |_{t=0}\\
&=\Delta'(\trace h)+\Delta(\trace'h)+\delta'(\delta h)+\delta(\delta' h) - \langle\ric',h\rangle - \langle\ric,h\rangle'\\
&=\Delta'(\trace h)-\Delta(|h|^2)+\delta'(\delta h)+\delta(\delta' h) - \langle\ric',h\rangle + \langle\ric,h\circ h\rangle,
\end{align*}
where we used standard variational formulas for the trace and the scalar product.
The proof of the lemma is completed by using the variational formulas
\begin{align*}
\ric'&=D_g\ric(h)=\frac{1}{2}\Delta_L h-\delta^*(\delta h)-\frac{1}{2}\nabla^2\trace h,\\
\Delta'f&= D_g\Delta(h)f=\langle h,\nabla^2f\rangle-\langle\delta h+\frac{1}{2}\nabla \trace h,\nabla f\rangle,\\
\delta'\omega&=D_g\delta(h)\omega=\langle h,\nabla\omega\rangle-\langle\delta h,\omega\rangle-\frac{1}{2}\langle\nabla\trace h,\omega\rangle,
\end{align*}
where $f$ is a function and $\omega$ a one-form. 
For more details on these formulas, see for example \cite[Section 1.K]{Besse} or the appendix of the second author's PhD thesis \cite{kroncke2013stability}.

\end{proof}

\section{Underdetermined elliptic equations}
\label{sec_analysis}

In our construction we will need to find solutions of underdetermined elliptic equations whose supports are contained in a prescribed compact set. In this section we review constructions and results from the work of Delay \cite{Delay2012} as we will apply them.

\subsection{Weighted function spaces}

We begin by introducing the weighted function spaces used by Delay in \cite{Delay2012}. Let $\Omega$ be a precompact open set with smooth boundary, 
and let $x$ be a smooth nonnegative defining function for the boundary, so that $\partial \Omega = x^{-1}(\{0\})$ and $dx\neq 0$ on $\partial \Omega$. For $a \in \N$ and $s\in \R$, $s \neq 0$, define the functions
\[
\phi := x^2,
\quad
\psi := x^{2(a-n/2)} e^{-s/x},
\quad
\varphi := x^{2a} e^{-s/x} .
\]
Define the weighted Sobolev spaces $\mathring{H}^k_{\phi,\psi}$, the weighted H\"older spaces $C^{k,\alpha}_{\phi,\varphi}$, and the weighted Fr\'echet space  $C^{\infty}_{\phi,\varphi}$ as in \cite{Delay2012}, that is
\[
\left\|u\right\|_{H^k_{\phi,\psi}}=\left(\int_\Omega \sum_{i=0}^k \phi^{2i}|\nabla^iu|^2\psi^2\dv \right)^{\frac{1}{2}} .
\]
As the choice of norm depends on the parameters $s$ and $a$ in an essential way, we introduce an alternative notation, which better fits in the framework of other conventions for weighted Sobolev spaces. Let $\theta=e^{-\frac{1}{x}}$. 
For $\delta,s\in\R$, we set
\[
\left\|u\right\|_{H^k_{\delta,s}}=\left(\int_\Omega \sum_{i=0}^k |\phi^{(i-\delta)}\nabla^iu|^2\theta^{-2s}\phi^{-n}\dv \right)^{\frac{1}{2}}
\]
and define $\mathring{H}^k_{\delta,s}$ to be the closure of $C^{\infty}_{c}$ with respect to this norm. Comparing these two notions, we easily see that 
\[
\left\|u\right\|_{H^k_{\phi,\psi}}=\left\|u\right\|_{H^k_{-a,-s}}.
\]
Note also that  $\mathring{H}^k_{\delta}:=\mathring{H}^k_{\delta,0}$ a standard example of a weighted Sobolev space. A convenient reference for their properties are the lecture notes \cite{baer-skript} of B\"{a}r. 

Let us collect a few properties of these spaces. 
At first, if $\delta_1\leq\delta_2$, $s_1\leq s_2$ and $k,l\in\N_0$, we have
\[
\mathring{H}^{k+l}_{\delta_2,s_2}\subset \mathring{H}^{k}_{\delta_1,s_1} .
\]
Furthermore for $k,l\in\N_0$, $\delta,\delta_1,s,s_1\in\R$, we have bounded maps
\[
\nabla^l:\mathring{H}^{k+l}_{\delta,s}\to \mathring{H}^{k}_{\delta-l,s}, 
\qquad
\phi^{\delta_1}:\mathring{H}^k_{\delta,s}\to \mathring{H}^k_{\delta+\delta_1,s},
\qquad
\theta^{s_1}:\mathring{H}^k_{\delta,s}\to \mathring{H}^k_{\delta,s+s_1},
\]
where the latter two are of course isomorphisms. The H\"older inequality in this setting states that 
\[
\left\|u v\right\|_{L^r_{\delta_1+\delta_2}}\leq C\left\|u\right\|_{L^p_{\delta_1}}\left\|v\right\|_{L^p_{\delta_2}},
\quad \text{if} \quad \frac{1}{r}=\frac{1}{p}+\frac{1}{q}, \delta_1,\delta_2\in\R
\]
and the Sobolev embedding theorem tells us that
\[
\left\|u\right\|_{W^{k+l,p}_{\delta}}\leq C\left\|u\right\|_{W^{k,q}_{\delta}},\quad \text{if} \quad \frac{1}{p}>\frac{1}{q}-\frac{l}{n}.
\]
The following lemma is a straightforward consequence.
\begin{lem}
If $k>n/2$ we have a continuous multiplication
\[
\mathring{H}^{k}_{\delta_1}\times \mathring{H}^k_{\delta_2}\to \mathring{H}^k_{\delta_1+\delta_2}.
\]	
In particular $\mathring{H}^{k}_{\delta}$ is then an algebra for $\delta\geq0$.
\end{lem}

\begin{proof}
For $0\leq j\leq l\leq k$, choose $p,q$ such that $\frac{1}{2}=\frac{1}{p}+\frac{1}{q}$ and $\frac{1}{p}>\frac{1}{2}-\frac{l-j}{n}$, $\frac{1}{q}>\frac{1}{2}-\frac{j}{n}$. This is possible since $k>n/2$. For $0 \leq l \leq k$ we estimate
\[\begin{split}
\left\| \nabla^l(u v)\right\|_{L^2_{\delta_1+\delta_2-l}}&\leq C\sum_{j=0}^l\left\|\nabla^{l-j}u\otimes \nabla^jv\right\|_{H^k_{\delta_1+\delta_2-l}}\\
&\leq C\sum_{j=0}^l\left\|\nabla^{l-j}u\right\|_{L^{p}_{\delta_1-(l-j)}}\left\| \nabla^jv\right\|_{L^{q}_{\delta_1-j}}\\
&\leq C\sum_{j=0}^l\left\|u\right\|_{W^{l-j,p}_{\delta_1}}\left\| v\right\|_{W^{j,q}_{\delta_1}}\\
&\leq C\sum_{j=0}^l\left\|u\right\|_{H^{k}_{\delta_1}}\left\| v\right\|_{H^{k}_{\delta_2}}.
\end{split}\]
\end{proof}

\begin{cor}
If $k>n/2$ we have a continuous multiplication
\[
\mathring{H}^{k}_{\delta_1,s_1}\times \mathring{H}^k_{\delta_2,s_2}\to \mathring{H}^k_{\delta_1+\delta_2,s_1+s_2}.
\]	
In particular $\mathring{H}^{k}_{\delta,s}$ is then an algebra for $\delta,s\geq0$.
\end{cor}

\begin{proof}
This follows directly from the above lemma, since
\begin{align*}
\mathring{H}^{k}_{\delta_1,s_1}&=\theta^{s_1}\mathring{H}^{k}_{\delta_1,0}=\theta^{s_1}\mathring{H}^{k}_{\delta_1},\\
\mathring{H}^{k}_{\delta_2,s_2}&=\theta^{s_2}\mathring{H}^{k}_{\delta_2,0}=\theta^{s_2}\mathring{H}^{k}_{\delta_2},\\
\mathring{H}^{k}_{\delta_1+\delta_2,s_1+s_2}&=\theta^{s_1+s_2}\mathring{H}^{k}_{\delta_1+\delta_2,0}=\theta^{s_1+s_2}\mathring{H}^{k}_{\delta_1+\delta_2}.
\end{align*}
\end{proof}

\begin{thm} \label{thmproductsofderivatives}
Let $\mathring{H}^k_{\phi,\psi}=\mathring{H}^k_{-a,-s}$ where $s>0$, $a\geq n$, and $k>\frac{n}{2}$. Then for
\[
u,v\in \psi^2\phi^2 \mathring{H}^{k+2}_{\phi,\psi},
\]
we have that
\[
\nabla^iu \otimes \nabla^jv \in \psi^2 \mathring{H}^{k}_{\phi,\psi}
\]
for $0 \leq i,j \leq 2$.
\end{thm}

\begin{rem}
This assertion extends easily to an arbitrary number of tensor products of the form $\nabla^{i_1}u_1\otimes\cdots\otimes\nabla^{i_N}u_N$. In particular, the scalar curvature map $g\mapsto \scal_g$ can be extended as a map between these spaces.
\end{rem}

\begin{proof}
Note that 
\begin{align*}
\psi^2\phi^2 \mathring{H}^{k+2}_{\phi,\psi}&=\theta^{2s}\phi^{2+2a-n} \mathring{H}^{k+2}_{-a,-s} = \mathring{H}^{k+2}_{a+2-n,s},\\
\psi^2\mathring{H}^{k}_{\phi,\psi}&=\theta^{2s}\phi^{2a-n} \mathring{H}^{k}_{-a,-s}=\mathring{H}^{k}_{a-n,s},
\end{align*}
and that
\[
\nabla^iu,\nabla^jv\in \mathring{H}^{k}_{a-n,s}.
\]
Since $k>\frac{n}{2}$ and $a-n\geq0$ the result follows by the above corollary.
\end{proof}
Note that $s<0$ gives spaces of sections which vanish to all orders at the boundary, whereas $s>0$ gives spaces of sections which blow up to all orders.
\begin{assm} \label{ass-s>0}
From now on, we assume $s<0$.
\end{assm}
\subsection{Trace-free symmetric two-tensors with prescribed double divergence}

We will now state the main result of \cite{Delay2012} in the case of the double
divergence operator $\delta^2$ acting on trace-free symmetric two-tensors. This will be applied in Section \ref{sec_prescribe_scal}

Let $P := \delta^2$ considered as a map from trace-free symmetric two-tensors to functions. The formal adjoint is then the trace-free part of the Hessian, $P^* = \mathring{\nabla}^2$. We make the following assumption.
\begin{assm} \label{assm-not-locally-warped}
We assume that the kernel $\mathcal{K}$ of $P^*$ consists only of constant functions.
\end{assm}
This assumption is equivalent to the metric not being locally a warped
product (see Theorem 4.3.3 of \cite{PetersenRiemannianGeometry}).

Since $P = \delta \circ \delta$ acting on trace-free symmetric two-tensors, we have
$P^* = \mathcal{D} \circ d$ where $d$ is the differential acting on functions
and $\mathcal{D}$ is the conformal Killing operator acting on 1-forms.
By Sections 9.1 and 9.4 in \cite{Delay2012}, the conditions (API) and (KRC)
hold for $d$ and $\mathcal{D}$ and we conclude that these conditions hold
for $P^*$ as well.

Define the operator
\[
\mathcal{L}_{\phi,\psi} := \psi^{-2} P \psi^2 \phi^4 P^*.
\]
Let $\mathcal{K}^\perp$ be the orthogonal complement of $\mathcal{K}$
with respect to the $L^2_\psi$ inner product. Then $\mathcal{K}^\perp$
consists of all functions $\tilde{f}$ such that $\int_{\Omega} \tilde{f}
\psi^2 \dv = 0$. Let $\pi_{\mathcal{K}^\perp}$ be the orthogonal
projection onto $\mathcal{K}^\perp$  with respect to the $L^2_\psi$
inner product. By Theorem 3.8 of \cite{Delay2012} we have that
\[
\pi_{\mathcal{K}^\perp} \mathcal{L}_{\phi,\psi}:
\mathcal{K}^\perp \cap \mathring{H}^{k+4}_{\phi,\psi}
\to
\mathcal{K}^\perp \cap \mathring{H}^{k}_{\phi,\psi}
\]
is an isomorphism. By (10.3) in \cite{Delay2012} we further have that
\[
\phi^2 P^* :
\mathring{H}^{k+4}_{\phi,\psi}
\to
\mathring{H}^{k+2}_{\phi,\psi}
\]
is a bounded map, which means that also
\[
\psi^2\phi^4 P^* :
\mathring{H}^{k+4}_{\phi,\psi}
\to
\psi^2\phi^2\mathring{H}^{k+2}_{\phi,\psi}
\]
is bounded. Define the operator $Q$ by
\begin{equation} \label{def-Q}
Q(f) = \psi^2\phi^4 P^* \left(
(\pi_{\mathcal{K}^\perp}\mathcal{L}_{\phi,\psi})^{-1} \left( \psi^{-2} f
\right)
\right).
\end{equation}
By the above, this is a bounded map
\[
Q:
\psi^2 \left( \mathcal{K}^\perp \cap \mathring{H}^{k}_{\phi,\psi} \right)
\to
\psi^2\phi^2\mathring{H}^{k+2}_{\phi,\psi} .
\]

Let $f \in \psi^2C^{\infty}_{\phi,\varphi}$ be such that $\int_\Omega f \dv = 0$. The trace-free symmetric two-tensor $U := Q(f)$ is a solution to the equation $P(U) = f$ and Theorem 5.1 of \cite{Delay2012} tells us that $U \in \psi^2\phi^2C^{\infty}_{\phi,\varphi}$ so that $U$ is smooth and $U$ and all its derivatives vanish on $\partial \Omega$. This theorem is formulated with the assumption that $f$ smooth with compact support in $\Omega$, but the proof only uses that $f \in \psi^2C^{\infty}_{\phi,\varphi}$.

We summarize the above discussion in the following theorem.

\begin{thm}[Delay] \label{thm-delay}
Let $(M,g)$ be a Riemannian manifold for which Assumption \ref{assm-not-locally-warped} holds.
For any $f \in \psi^2C^{\infty}_{\phi,\varphi}(M)$ with $\int_\Omega f \dv
= 0$, the trace-free two-tensor $U := Q(f)$ is a solution to
\[
\delta(\delta U) = f
\]
with $U \in \psi^2\phi^2C^{\infty}_{\phi,\varphi}(S^2_0M)$ so that $U$ is smooth
and $U$ and all its derivatives vanish on $\partial \Omega$.
Further, for any nonnegative integer $k$ there is a constant $C$ so that
\begin{equation} \label{estimate-delay}
\| U  \|_{\psi^2\phi^2\mathring{H}^{k+2}_{\phi,\psi}}
\leq
C \| f \|_{\psi^2 \mathring{H}^{k}_{\phi,\psi} }
\end{equation}
for all such $f$.
\end{thm}

\subsection{Density of TT-tensors}

We will now apply the results of \cite{Delay2012} to conclude that the compactly supported TT-tensors are dense in spaces of TT-tensors of some specific regularity. This will be applied in Section \ref{sec_prescribe_scal} and Section \ref{sec_local_rigidity}.

\begin{lem}\label{lem:TTdensity}
Let $(M,g)$ be a compact manifold with boundary.
Then the space $C^{\infty}_c(TT)$ of compactly supported $TT$-tensors is $\psi^2\phi^2 H^k_{\phi,\psi}$-dense in $ \psi^2\phi^2 \mathring{H}^k_{\phi,\psi}(TT)$.\end{lem}

\begin{proof}
Let $h\in TT\cap \psi^2\phi^2 \mathring{H}^k_{\phi,\psi}(S^2_0M)$. By definition, there is a sequence $h_i\in C^{\infty}_{c}(S^2_0M)$ such that $h_i\to h$ in the $\psi^2\phi^2 H^k_{\phi,\psi}$-norm. 
Now we would like to find correction tensors $k_i\in C^{\infty}_{c}(S^2_0M)$ such that $h_i-k_i\to h$ in the $\psi^2\phi^2 H^k_{\phi,\psi}$-norm and $\delta(h_i-k_i)=0$. We therefore need to solve the equation
\[
\delta k_i=\delta h_i=\delta (h_i-h).
\]
Let $a\in\mathbb{N}$ and $s>0$ be the parameters such that $\mathring{H}^k_{-a,-s} = \mathring{H}^k_{\phi,\psi}$. Then $\psi^2\phi^2 \mathring{H}^k_{\phi,\psi} = \mathring{H}^k_{a+2-n,s}$ and $\delta (h_i-h)\to 0$  in the $H^{k-1}_{a+1-n,s}$-norm.

We now apply Delay's results with $\hat{P}=\delta$ acting on trace-free symmetric two-tensors, so that $\hat{P}^*= \mathcal{D}$ where $\mathcal{D}$ is the conformal Killing operator acting on 1-forms. The kernel $\mathcal{K}$ of $\mathcal{D}$ consists of the 1-forms dual to conformal Killing fields. 
Let $\mathcal{K}^\perp$ be the orthogonal complement of $\mathcal{K}$
with respect to the $L^2_\psi$ inner product and let $\pi_{\mathcal{K}^\perp}$ be the orthogonal projection onto $\mathcal{K}^\perp$.
Define the operator 
\[
\mathcal{L}_{\phi,\psi} :=\psi^{-2}\hat{P}\psi^2\phi^2\hat{P}^*: 
\mathring{H}^{k+1}_{-b,-s} \to \mathring{H}^{k-1}_{-b,-s}.
\]
By Theorem 3.8 of \cite{Delay2012} we know that 
\[
\pi_{\mathcal{K}^\perp} \mathcal{L}_{\phi,\psi}:
\mathcal{K}^\perp \cap \mathring{H}^{k+1}_{-b,-s}
\to \mathcal{K}^\perp \cap \mathring{H}^{k-1}_{-b,-s}
\]
is an isomorphism for each $b\in\N$ whose inverse we denote by $\hat{Q}$.

Since $h_i-h$ vanishes on the boundary of $M$ we have $\psi^{-2}\delta (h_i-h) \in \mathcal{K}^\perp$. The solution $k_i$ of the equation
\[
\delta k_i=\delta (h_i-h)
\]
is thus given by
\begin{equation}\label{find k_i}
k_i = \psi^2\phi^2\hat{P}^*\hat{Q}\psi^{-2}\delta (h_i-h).
\end{equation}
We have $\psi^{-2}\delta (h_i-h)\to 0$ in $H^{k-1}_{-a+1,-s}$. If we set $b=a-1$ in the above isomorphism, we get $\hat{Q}\psi^{-2}\delta (h_i-h)\to 0$ in $H^{k+1}_{-a+1,-s}$ and $\hat{P}^*\hat{Q}\psi^{-2}\delta (h_i-h)\to 0$ in $H^{k}_{-a,s}=H^k_{\phi,\psi}$
so that finally
\[
k_i=\psi^2\phi^2\hat{P}^*\hat{Q}\psi^{-2}\delta (h_i-h)\to 0
\]
in $\psi^2\phi^2H^k_{\phi,\psi}$ which is what we wanted to prove.

To ensure that the $k_i$ are compactly supported in the interior of $M$, we refine the argument as follows.
 Take a sequence of diffeomorphisms $\Phi_i:M\to \Omega_i\subset M$ converging uniformly in all derivatives to the identity map on $M$, where the precompact subsets 
 $\Omega_i$ are such that $\mathrm{supp}(h_i)$ is compactly supported in the interior of $\Omega_i$. Such diffeomorphisms can be constructed by choosing a tubular neighborhood of $\partial M$ and deforming the identity map in the radial direction. 
%
We then get that $g_i:=\Phi_i^*g$ converges to $g$ smoothly in all derivatives.
Let $\tilde{h}_i=\Phi_i^*h_i$.
We define $\tilde{k}_i$, similar to $k_i$ in \eqref{find k_i}, by 
\begin{equation}\label{find tildek_i}
\tilde{k}_i := \psi^2\phi^2\hat{P}_i^*\hat{Q}_i\psi^{-2}\delta^{g_i} (\tilde{h}_i-\Phi_i^*\tilde{h}),
\end{equation}
where
the operators $\hat{P}_i$ and $\hat{Q}_i$ are now defined with respect to the metric $g_i$. Note that with the same parameters $a,s$ as above, we have $\Phi_i^*h\notin \psi^2\phi^2 \mathring{H}^k_{\phi,\psi} = \mathring{H}^k_{a+2-n,s}$, because $h$ is not supported in $\Omega_i$. However, since $\delta^{g_i}(\Phi_i^*\tilde{h})=0$ and the $\tilde{h}_i$ are supported in the interior of $M$, we still have 
$\delta^{g_i} (\tilde{h}_i-\Phi_i^*\tilde{h})\in H^{k-1}_{a+1-n,s}$, which is what we need to make sense of $\tilde{k}_i$ in \eqref{find tildek_i}. By construction,
$\delta^{g_i}(\tilde{h}_i-\tilde{k_i})=0$.
Since $g_i\to g$ smoothly, $\delta^{g_i} (\tilde{h}_i-\Phi_i^*\tilde{h})\to 0$ in $ H^{k-1}_{a+1-n,s}$
and $\tilde{k}_i \to 0$ in $\psi^2\phi^2 \mathring{H}^k_{\phi,\psi}$.
Now the tensors $k_i=(\Phi_i)_*\tilde{k}_i$ are supported in $\Omega_i$, and so are $h_i-k_i$. By diffeomorphism invariance of the divergence, we get $\delta^g(h_i-k_i)=0$ and because $\Phi_i$ converges to the identity, we also get $k_i\to 0$ in $\psi^2\phi^2 \mathring{H}^k_{\phi,\psi}$. Therefore, $h_i-k_i\to h$ in $\psi^2\phi^2 \mathring{H}^k_{\phi,\psi}$, as desired.
\end{proof}

\begin{lem}\label{lem:TTdensity2}
Let $(M,g)$ be a compact manifold with boundary. Then the space $C^{\infty}_c(TT)$ of compactly supported $TT$-tensors is $H^1$-dense in $\mathring{H}^1(TT)$.
\end{lem}

\begin{proof}
We deduce this density lemma from Lemma \ref{lem:TTdensity} and a trick. Consider $\delta$ as a map $C^{\infty}(S^2_0M)\to C^{\infty}(T^*M)$ and consider its formal adjoint
\[
\mathcal{D}
=\left(1-\frac{1}{n}g\cdot\mathrm{tr}\right)\circ \delta^*:C^{\infty}(T^*M)\to C^{\infty}(S^2_0M).
\]
Then we have the $L^2$-orthogonal decomposition
\[
\psi^2\phi^2 \mathring{H}^k_{\phi,\psi}(S^2_0M) =
\overline{\mathcal{D} (C^{\infty}_c(T^*M))}^{\psi^2\phi^2H^k_{\phi,\psi}}\oplus\psi^2\phi^2 \mathring{H}^k_{\phi,\psi}(TT)
\]
and by Lemma \ref{lem:TTdensity}, we actually have 
\[
\psi^2\phi^2 \mathring{H}^k_{\phi,\psi}(S^2_0M)
=\overline{\mathcal{D} (C^{\infty}_c(T^*M))}^{\psi^2\phi^2H^k_{\phi,\psi}}\oplus 
\overline{C^{\infty}_c(TT)}^{\psi^2\phi^2H^k_{\phi,\psi}}.
\]
Now if we take the closure of this decomposition in the $H^1$-norm, we get
\[
\mathring{H}^1(S^2_0M) = 
\overline{\mathcal{D}(C^{\infty}_c(T^*M))}^{H^1}
\oplus \overline{C^{\infty}_c(TT)}^{H^1}.
\]
Let us justify this statement: Let $U\subset \mathring{H}^1(S^2_0M)$ be a dense subspace which admits the $L^2$-orthogonal decomposition $U = V \oplus W$.
Further, let $u \in \overline{U}^{H^1}=\mathring{H}^1(S^2_0M)$ and $u_i$ be a sequence in $U$, uniquely decomposed as $u_i=v_i+w_i$ with $v_i \in V$, $w_i \in W$. since the $L^2$-orthogonal projections
\[
\pi_V:\mathring{H}^1(S^2_0M)\to \overline{V}^{H^1}, \qquad
\pi_W:\mathring{H}^1(S^2_0M)\to \overline{W}^{H^1}
\]
are $H^1$-continuous, $\pi_V(u_i)=v_i$ and $\pi_W(u_i)=w_i$ converge in $H^1$ to limits $v\in \overline{V}^{H^1}$ and $w\in\overline{W}^{H^1}$ respectively and we have $u=v+w$. Clearly $ \overline{V}^{H^1}$ and $\overline{W}^{H^1}$ have trivial intersection as they are $L^2$-orthogonal as well. Thus
\[
\overline{U}^{H^1}=\mathring{H}^1(S^2_0M)=\overline{V}^{H^1}\oplus \overline{W}^{H^1},
\]
which is what we stated. Since we of course also have the decompositon
\[
\mathring{H}^1(S^2_0M)=
\overline{\mathcal{D}(C^{\infty}_c(T^*M)}^{H^1}\oplus \mathring{H}^1(TT),
\]
we may conclude 
\[
\overline{C^{\infty}_c(TT)}^{H^1} = \mathring{H}^1(TT),
\]
as desired.
\end{proof}

Note that if $(M,g)$ is an Einstein manifold with Einstein constant $\sigma$, we have the standard commutation formulas (see for example \cite{Lichnerowicz1961} and \cite[p.~15]{kroncke2013stability})
\begin{align*}
\trace\circ \Delta_E    &=(\Delta-2\sigma)\circ\trace,\\
\delta\circ \Delta_E    &=(\Delta_H-2\sigma)\circ\delta,\\
\Delta_E\circ\delta^*   &=\delta^*\circ(\Delta_H-2\sigma),\\
\Delta_E(fg)            &=(\Delta f-2\sigma f) g,
\end{align*}
where $\Delta_H$ is the Hodge Laplacian on one-forms and $f$ is a function.
Hence, $\Delta_E$ is diagonal with respect to the $L^2$-orthogonal splitting
\[
\mathring{H}^1(S^2M) 
=
\left(\mathring{H}^1(M)\cdot g + \overline{\delta^*(C^{\infty}_c(T^*M))}^{H^1}\right)
\oplus \mathring{H}^1(TT).
\]
Thus, it makes sense to speak about the lowest (Dirichlet) eigenvalue of $\Delta_E$ on TT-tensors.

\begin{cor} \label{cor-neg-cpt-supp}
Let $M$ be a compact Einstein manifold with boundary. Then its interior is linearly stable in the sense of Definition \ref{def-stable-open} if and only if the smallest Dirichlet eigenvalue of $\Delta_E$ on TT-tensors is negative.
\end{cor}

\begin{proof}
By Lemma \ref{lem:TTdensity2} and continuity, we have
\begin{align*}
&\inf\left\{(\Delta_Eh,h)_{L^2}\mid h\in C^{\infty}_c(TT),\left\|h\right\|^2_{L^2}=1\right\}\\
&\qquad=
\inf\left\{(\Delta_Eh,h)_{L^2}\mid h\in \mathring{H}^1(TT),\left\|h\right\|^2_{L^2}=1\right\},
\end{align*}
which immediately implies the result.
\end{proof}

\begin{lem} \label{lemma_h+-}
Let $h\in C^{\infty}_{c}(TT)$ satisfy $h\neq0$ and suppose that
\[
\int_M \langle \Delta_E h,h\rangle \dv=0.
\]
Then 
there exist tensors $h_{\pm}\in C^{\infty}_{c}(\Omega)\cap TT$ such that
\begin{equation} \label{hpm}
\int_M \langle \Delta_Eh_{\pm},h_{\pm}\rangle\mathrm{dV}=\pm1,\qquad
\int_M \langle \Delta_Eh_{+},h_{-}\rangle\mathrm{dV}=0.
\end{equation}
and $h=\frac{1}{2}(h_++h_-)$.
\end{lem}
\begin{rem}
This result reflects the following elementary fact: a null vector for a non-degenerate quadratic form can be written as the sum of normalized orthogonal spacelike and timelike vectors. This fact is needed in an essential way in a construction in the next chapter.
\end{rem}	
	
\begin{proof}
Consider the symmetric bilinear form $Q:C^{\infty}_{c}(TT)\times C^{\infty}_{c}(TT)\to\R$, given by
\[
Q(k_1,k_2) = \int_M \langle \Delta_E k_1,k_2\rangle\mathrm{dV}.
\]
Observe that it is non-degenerate since $\Delta_E:C^{\infty}_{c}(TT)\to C^{\infty}_{c}(TT)$ does not have a kernel.
We first prove that there exists a $k\in C^{\infty}_{c}(TT)$ such that $Q(k,k)=0$ and $Q(h,k)=\frac{1}{2}$. To find this $k$, let $k_0$ be such that $\alpha:=Q(k_0,h)\neq 0$. If $\beta:=Q(k_0,k_0)=0$, we set $k=\frac{1}{2\alpha}k_0$. If $\beta\neq 0$, we observe that $k_1=-\frac{2\alpha}{\beta}k_0+h$ satisfies $Q(k_1,k_1)=0$ and $Q(k_1,h)=-\frac{2\alpha}{\beta}Q(k_0,h)=-\frac{2\alpha^2}{\beta}$. Then $k=-\frac{\beta}{4\alpha^2}k_1$ has the required properties. 
It is then easy to see that the tensors $h_{\pm}=h\pm k$ have the properties stated in the lemma.
\end{proof}

\section{Prescribing scalar curvatures}
\label{sec_prescribe_scal}

In our next theorem we will construct 1-parameter deformations of a metric $g$ with prescribed scalar curvature and volume form, and a symmetric two-tensor $h$ prescribed as the first derivative of the deformation.

\begin{thm} \label{thm-prescribe-scal}
Assume that $(M,g)$ is a compact Einstein manifold with boundary, which is not locally a warped product. Let $f_t$ be a $1$-parameter family of smooth functions, with $C^3$-regularity in the parameter $t$. Assume that the $f_t$ are supported in the open set $\Omega$ which is relatively compact in the interior of $M$. Further, assume that $h \neq 0$ is a smooth TT-tensor with support in $\Omega$ satisfying 
\begin{equation} \label{assumption-h}
\int_M\langle \Delta_E h, h \rangle \dv= -2 \int_M f_0 \dv.
\end{equation}
Then there exists a $1$-parameter family $g_t$ of metrics with $g_0=g$
and $\frac{d}{dt}g_t |_{t=0} = h$ such that 
\begin{equation} \label{scal=f_t}
\scal^{g_t} = \scal^g + \frac{t^2}{2}f_t,
\end{equation}
\begin{equation} \label{dV-preserved}
\dv^{g_t}=\dv^{g},
\end{equation}
and $g_t = g$ outside of $\Omega$.
\end{thm}
\begin{rem}
A famous result by Kazdan-Warner \cite{KW75} (with an improvement by B{\'e}rard Bergery \cite{BB81}) asserts that closed manifolds divide into three disjoint classes according to which functions can be the scalar curvature of a Riemannian metric:
\begin{itemize}
\item[(i)] Any smooth function is the scalar curvature of a smooth metric.
\item[(ii)] A smooth function is the scalar curvature of a smooth metric if and only if it is either identically zero or strictly negative somewhere. In this case, any scalar-flat metric is Ricci-flat.
\item[(iii)] A smooth function is the scalar curvature of a smooth metric if and only if it is strictly negative somewhere.
\end{itemize}
Similarly, Theorem \ref{thm-prescribe-scal} gives us a dichotomy for perturbations of the scalar curvature:
\begin{itemize}
\item[(i)'] If $(M,g)$ is linearly unstable, any $C^{\infty}_c$-function is a second order scalar curvature perturbation of a $C^{\infty}_c(TT)$-perturbation of $g$.
\item[(iii)'] If $(M,g)$ is linearly stable, a $C^{\infty}_c$-function is a second order scalar curvature perturbation of a $C^{\infty}_c(TT)$-perturbation of $g$ if and only if it has negative integral.
\end{itemize}
An analogue of (ii) does not exist: This would correspond to the \emph{neutrally lineary stable} case where $\mu_1^{D}(\Delta_E,TT)=0$. However, the infimum in $\mu_1^{D}(\Delta_E,TT)$ can then not be realized by a $C^{\infty}_c$-tensor.
\end{rem}

\begin{proof}[{Proof of Theorem \ref{thm-prescribe-scal}}]
The assertion can be regarded as a kind of second order implicit function theorem. As in the standard implicit function theorem, the proof is based on a contraction argument, but in this case, it is for second order perturbations. We work in Delay's function spaces $\psi^2\phi^2 \mathring{H}^{k+2}_{\phi,\psi}$ we introduced in Section \ref{sec_analysis}.
To carefully execute the arguments, we have divided the proof into seven steps. Throughout the proof, the Einstein constant of $g$ is be denoted by $\sigma$.

{\bf Step 1: Solving to second order at $t=0$.} 
Let $h$ be a smooth TT-tensor such that \eqref{assumption-h} holds and let $k$ be a symmetric two-tensor. Set $g_t := g + th + \frac{t^2}{2}k$.
Since $h$ is a TT-tensor we have 
\[
\frac{d}{dt} \scal^{g_{t}} |_{t=0} = D_g\scal(h) = 0
\]
and 
\[
\frac{d}{dt} dV^{g_{t}} |_{t=0} = 
D_g dV (h) = \left(\frac{1}{2} \trace^g h \right) dV^g = 0.
\]
 We want to choose $k$ such that
\begin{equation}\label{eq:second_order_scal}
\frac{d^2}{dt^2} \scal^{g_{t}} |_{t=0} = 
D_g\scal(k) + D^2_g\scal(h,h) = f_0,
\end{equation}
and
\begin{equation}\label{eq:second_order_vol}
\frac{d^2}{dt^2} dV^{g_{t}} |_{t=0} = 
D_g dV(k) + D^2_g dV(h,h) = 0,
\end{equation}
so that $\scal^{g_t} = \scal^{g} + \frac{t^2}{2}f_0 + O(t^3)$ and 
$dV^{g_t} = dV^g + O(t^3)$. 
First, set $k= \frac{1}{n}|h|^2g + \mathring{k}$, where $\mathring{k}$ is a trace-free two-tensor. Then,
\[ 
D_g dV(k) + D^2_g dV(h,h) = 
\left(\frac{1}{2} \trace^g k + \frac{1}{4} (\trace^g h)^2 - \frac{1}{2}|h|^2_g \right) dV^g = 0,
\]
and \eqref{eq:second_order_vol} holds. To solve \eqref{eq:second_order_scal}, we recall from Lemmas \ref{lem_first_variation_scal} and \ref{lem_second_variation_scal} that 
\begin{align*}
D_g\scal(k) &= \Delta(\trace^g k) + \delta(\delta k) -\sigma \trace^g k, \\
D^2_g\scal(h,h) &=
-\Delta(|h|^2)+\delta(\delta' h)-\frac{1}{2}\langle\Delta_E h,h\rangle
+\sigma|h|^2,
\end{align*}
where the operator $\delta'$ also depends linearly on $h$.
We now have
\[ 
D_g\scal(k) + D^2_g\scal(h,h)
=
\delta
\left( \delta' h -\frac{1}{n}d|h|^2 + \delta \mathring{k} \right) 
-\frac{1}{2}\langle\Delta_E h,h\rangle.
\]
The function
\[
-\delta \left( \delta'h - \frac{1}{n}d|h|^2 \right)
\]
has vanishing integral, and by assumption the same holds for
\[
\frac{1}{2}\langle \Delta_E h,h\rangle + f_0,
\]
so from Theorem \ref{thm-delay} we get a trace-free symmetric two-tensor $\mathring{k}$ with support in $\Omega$ such that
\[
\delta\delta \mathring{k}
= -\delta\left(\delta'h-\frac{1}{n}d(|h|^2) \right)
+ \frac{1}{2}\langle \Delta_E h,h\rangle + f_0.
\]
With this choice of $\mathring{k}$, our tensor $k\in C^{\infty}(S^2 M)$ satisfies \eqref{eq:second_order_scal} and the first step is finished.

{\bf Step 2: Setting up an iteration.} 
Using an iteration argument, we are going deform the family $g_t$ to a solution of \eqref{scal=f_t} and \eqref{dV-preserved}.
	
For integers $i \geq 0$ set 
\[
g^{(i)}_t := g + h^{(i)}_t + \frac{1}{2}k^{(i)}_t,
\]
where $h^{(i)}_t, k^{(i)}_t$ are families of symmetric two-tensors depending on $t$, and we assume that the $h^{(i)}_t$ are TT-tensors.
The iteration begins with $h^{(0)}_t := th$ and $k^{(0)}_t := t^2 k$ so that $g^{(0)}_t = g_t$.
We will then find $h^{(i)}_t, k^{(i)}_t$ iteratively from the equations 
\begin{equation} \label{iteration-h-k}
\begin{split}
&D_g\scal(k^{(i+1)}_t-k^{(i)}_t) + D^2_g\scal(h^{(i+1)}_t,h^{(i+1)}_t) 
- D^2_g\scal(h^{(i)}_t,h^{(i)}_t) \\	
&\qquad
= \scal(g)+\frac{t^2}{2}f_t - \scal(g^{(i)}_t)
\end{split}
\end{equation}
and
\begin{equation} \label{iteration-dV}
\begin{split}
&D_g dV(k^{(i+1)}_t-k^{(i)}_t) + D^2_g dV(h^{(i+1)}_t,h^{(i+1)}_t) 
- D^2_g dV(h^{(i)}_t,h^{(i)}_t) \\	
&\qquad
= dV(g) - dV({g^{(i)}_t}) \\
&\qquad
= (1 - U^{(i)}_t) dV^g,  \\
\end{split}
\end{equation}
where $U^{(i)}_t dV(g) := dV(g^{(i)}_t)$. 

{\bf Step 3: Solving the iteration.} 
We now explain how to determine $h^{(i)}_t, k^{(i)}_t$. 

For the TT-tensor $h^{(i)}_t$, we use  two slightly different constructions depending on whether $F_0 := \int_N f_0 \dv$ vanishes or not. If
$F_0 \neq 0$ we set
\begin{equation} \label{def-h_t-1}
h^{(i)}_t := t (1+\lambda_t^{(i)}) h,
\end{equation}
where $\lambda_t^{(i)}$ is a function of $t$ only.
In the case $F_0 = 0$ we set  
\begin{equation} \label{def-h_t-2}
h^{(i)}_t := 
t\left(\frac{1}{2}(h_+ + h_-) + \lambda_t^{(i)}(h_+ - h_-) \right),
\end{equation}
where $\lambda_t^{(i)}$ is a function of $t$, and the compactly supported TT-tensors $h_+$ and $h_-$ are given by Lemma \ref{lemma_h+-}. The right choice of $\lambda_t^{(i)}$ will guarantee the solvability of \eqref{iteration-h-k}.
Further, we write 
$k^{(i)}_t = \frac{1}{n}(\trace^g k^{(i)}_t) g + \mathring{k}^{(i)}_t$ 
where $\trace^g \mathring{k}^{(i)}_t = 0$.
Suppose now that $h^{(i)}_t$ and $k^{(i)}_t$ are already obtained.
Demanding \eqref{iteration-dV} yields
\begin{equation} \label{prescribe-tr-k-i}
\frac{1}{2} \left(\trace^g k^{(i+1)}_t - |h^{(i+1)}_t|^2_g \right) 
- \frac{1}{2} \left(\trace^g k^{(i)}_t - |h^{(i)}_t|^2_g \right) \\
= 1 - U^{(i)}_t .
\end{equation}
Let us assume for the moment that this relation holds.
Then by integrating equation \eqref{iteration-h-k} and using \eqref{prescribe-tr-k-i} we get
\[ \begin{split}
&\int_M \left( \frac{t^2}{2}f_t - \scal(g^{(i)}_t) \right) \dv^g \\
&\qquad=
\int_M \left(
D_g\scal(k^{(i+1)}_t-k^{(i)}_t) + D^2_g\scal(h^{(i+1)}_t,h^{(i+1)}_t) 
- D^2_g\scal(h^{(i)}_t,h^{(i)}_t)
\right) \dv^g \\
&\qquad=
\int_M \left(
(\Delta - \sigma) \left( \trace^g k^{(i+1)}_t - |h^{(i+1)}_t|^2_g \right)
+ \delta(\delta k^{(i+1)}_t + \delta' h^{(i+1)}_t)
- \frac{1}{2}\langle \Delta_E h^{(i+1)}_t,h^{(i+1)}_t\rangle
\right) \dv^g \\
&\qquad \qquad -
\int_M \left(
(\Delta - \sigma) \left( \trace^g k^{(i)}_t - |h^{(i)}_t|^2_g \right)
+ \delta(\delta k^{(i)}_t + \delta' h^{(i)}_t)
- \frac{1}{2}\langle \Delta_E h^{(i)}_t,h^{(i)}_t\rangle
\right) \dv^g \\
&\qquad=
\int_M \left(
2(\Delta - \sigma) \left( 1 - U^{(i)}_t \right)
- \frac{1}{2} \left(
\langle \Delta_E h^{(i+1)}_t,h^{(i+1)}_t\rangle 
- \langle \Delta_E h^{(i)}_t,h^{(i)}_t\rangle
\right)
\right) \dv^g,
\end{split} \]
so
\begin{equation} \label{def-lambdat} 
\begin{split}
&\int_M \left( \frac{t^2}{2}f_t - \scal(g^{(i)}_t) 
- 2(\Delta - \sigma) \left( 1 - U^{(i)}_t \right) \right) \dv^g \\
&\qquad=
- \frac{1}{2} \int_M \left(
\langle \Delta_E h^{(i+1)}_t,h^{(i+1)}_t\rangle 
- \langle \Delta_E h^{(i)}_t,h^{(i)}_t\rangle
\right) \dv^g \\
&\qquad=
\begin{cases} 
F_0 t^2 \left( (1+\lambda_t^{(i+1)})^2 - (1+\lambda_t^{(i)})^2 \right),
&\text{if } F_0 \neq 0 ,\\
- t^2 \left(\lambda_t^{(i+1)} - \lambda_t^{(i)} \right),
&\text{if } F_0 = 0.
\end{cases}
\end{split} 
\end{equation}
 Provided that the involved quantities are sufficiently small, this yields a unique choice for $\lambda_t^{(i+1)}$ and hence for $h_t^{(i+1)}$. Then, $\trace^g k^{(i+1)}_t$ is defined by \eqref{prescribe-tr-k-i}.

Now it remains to determine the tracefree part $\mathring{k}^{(i+1)}_t$ of $k^{(i+1)}_t$.
For this, we consider
equation \eqref{iteration-h-k} which becomes
\[ \begin{split}
&\frac{t^2}{2}f_t - \scal(g^{(i)}_t) \\
&\qquad=
D_g\scal(k^{(i+1)}_t-k^{(i)}_t) + D^2_g\scal(h^{(i+1)}_t,h^{(i+1)}_t) 
- D^2_g\scal(h^{(i)}_t,h^{(i)}_t) \\
&\qquad=
2(\Delta - \sigma) \left( 1 - U^{(i)}_t \right)
- \frac{1}{2} \left(
\langle \Delta_E h^{(i+1)}_t,h^{(i+1)}_t\rangle 
- \langle \Delta_E h^{(i)}_t,h^{(i)}_t\rangle
\right) \\
&\qquad\qquad + \delta(\delta k^{(i+1)}_t + \delta' h^{(i+1)}_t)
- \delta(\delta k^{(i)}_t + \delta' h^{(i)}_t) \\
\end{split} \]
or
\[
\delta\delta \left( \mathring{k}^{(i+1)}_t - \mathring{k}^{(i)}_t \right) 
= \mathring{f}^{(i+1)}_t, 
\]
where
\[ \begin{split}
\mathring{f}^{(i+1)}_t 
&:=
\frac{t^2}{2}f_t - \scal(g^{(i)}_t) 
- 2(\Delta - \sigma) \left( 1 - U^{(i)}_t \right) \\
&\qquad
+ \frac{1}{2} \left(
\langle \Delta_E h^{(i+1)}_t,h^{(i+1)}_t\rangle 
- \langle \Delta_E h^{(i)}_t,h^{(i)}_t\rangle
\right) \\
&\qquad
- \delta(\frac{1}{n} d \trace^g k^{(i+1)}_t + \delta' h^{(i+1)}_t)
+ \delta(\frac{1}{n} d \trace^g k^{(i)}_t + \delta' h^{(i)}_t).
\end{split} \]
By the previous choices, $\mathring{f}^{(i+1)}_t$ integrates to zero.
We define  $\mathring{k}^{(i+1)}_t$ to be the trace-free symmetric two-tensor with support in $\Omega$ satisfying
\begin{equation} \label{def-ringkt} 
\mathring{k}^{(i+1)}_t - \mathring{k}^{(i)}_t 
= Q(\mathring{f}^{(i+1)}_t),
\end{equation}
where the operator $Q$ is defined in \eqref{def-Q}.

Summing up, we used equations \eqref{prescribe-tr-k-i}, \eqref{def-lambdat} and 
\eqref{def-h_t-1}, \eqref{def-h_t-2}, followed by Equation \eqref{def-ringkt} to compute $h^{(i+1)}_t$ and $k^{(i+1)}_t$ in terms of $h^{(i)}_t$ and $k^{(i)}_t$. We will show that for small $t$, this iteration procedure defines a contraction. 

{\bf Step 4: Rewriting the iteration.} 
We continue by rewriting Equations \eqref{iteration-h-k} and \eqref{iteration-dV}.
By exploiting the Taylor expansion formula, we will sum up some of the terms in these equations to small integral error term which are more convenient for establishing the contraction property. For this purpose, set 
\[
g^{(i)}_{t,s} := g + sh^{(i)}_t + \frac{s^2}{2}k^{(i)}_t.
\]
Then 
\[
\frac{d}{ds} g^{(i)}_{t,s} |_{s=0} = h^{(i)}_t, 
\quad 
\frac{d^2}{ds^2} g^{(i)}_{t,s} |_{s=0} = k^{(i)}_t, 
\]
and
\[
\frac{d}{ds} \scal(g^{(i)}_{t,s}) |_{s=0} = 0, 
\]
since $h^{(i)}_t$ is a TT-tensor. Next,
\[
\frac{d^2}{ds^2} \scal(g^{(i)}_{t,s})  |_{s=0} = 
D_g\scal(k^{(i)}_t) + D^2_g\scal(h^{(i)}_t,h^{(i)}_t).
\]
Taylor expansion gives us that the scalar curvature term in Equation \eqref{iteration-h-k} is
\[ \begin{split}
\scal(g^{(i)}_t)
&= \scal(g^{(i)}_{t,1}) \\
&= \scal(g) + D_g\scal(k^{(i)}_t) + D^2_g\scal(h^{(i)}_t,h^{(i)}_t) \\
&\qquad\qquad 
+\frac{1}{2} \int_0^1 (1-s)^2 \frac{d^3}{ds^3} \scal(g^{(i)}_{t,s}) \,ds,
\end{split}\]
so Equation \eqref{iteration-h-k} can be written as
\begin{equation} \label{iteration-h-k-2}
D_g\scal(k^{(i+1)}_t) + D^2_g\scal(h^{(i+1)}_t,h^{(i+1)}_t)
= \frac{t^2}{2}f_t - \frac{1}{2} S^{(i)}_t,
\end{equation}
where 
\[
S^{(i)}_t := \int_0^1 (1-s)^2 \frac{d^3}{ds^3} 
\scal\left(g + sh^{(i)}_t + \frac{s^2}{2}k^{(i)}_t \right) \, ds.
\]
In the same way, we have
\[
\frac{d}{ds} dV({g^{(i)}_{t,s}}) |_{s=0} = 0, 
\]
and
\[
\frac{d^2}{ds^2} dV({g^{(i)}_{t,s}}) |_{s=0} = 
D_g dV(k^{(i)}_t) + D^2_g dV(h^{(i)}_t,h^{(i)}_t),
\]
so Taylor expansion gives us 
\[ \begin{split}
dV({g^{(i)}_t})
&= dV({g^{(i)}_{t,1}}) \\
&= dV(g)
+ D_g dV(k^{(i)}_t) + D^2_g dV(h^{(i)}_t,h^{(i)}_t)
+ \frac{1}{2} \int_0^1 (1-s)^2 \frac{d^3}{ds^3} dV(g^{(i)}_{t,s}) \,ds,
\end{split}\]
and Equation \eqref{iteration-dV} can be written as
\begin{equation} \label{iteration-dV-2}
D_g dV(k^{(i+1)}_t) + D^2_g dV(h^{(i+1)}_t,h^{(i+1)}_t) 
= - \frac{1}{2} V^{(i)}_t dV(g),
\end{equation}
where 
\[
V^{(i)}_t dV(g) := \int_0^1 (1-s)^2 \frac{d^3}{ds^3} 
dV(g + sh^{(i)}_t + \frac{s^2}{2}k^{(i)}_t) \, ds.
\]
Let us now turn these formulas into definitions for $k^{(i+1)}$ and $h^{(i+1)}$, where we again use the splitting 
$k^{(i)}_t = \frac{1}{n}(\trace^g k^{(i)}_t) g + \mathring{k}^{(i)}_t$ into the pure trace part and the tracefree part.

For this purpose, observe first that \eqref{iteration-dV-2} tells us that 
\[\begin{split}
D_g dV(k^{(i+1)}_t) + D^2_g dV(h^{(i+1)}_t,h^{(i+1)}_t)  
&= \left(\frac{1}{2} \trace^g k^{(i+1)}_t 
- \frac{1}{2}|h^{(i+1)}_t|^2_g \right) dV(g) \\
&= - \frac{1}{2} V^{(i)}_t dV(g),
\end{split}\]
and we find that 
\begin{equation} \label{prescribe-tr-k-i-2} 
\trace^g k^{(i+1)}_t = |h^{(i+1)}_t|^2_g - V^{(i)}_t .
\end{equation}
Let us assume that this holds for the moment. Then by integrating Equation \eqref{iteration-h-k-2} and using \eqref{prescribe-tr-k-i-2} we get
\[ \begin{split}
&\int_M \left( \frac{t^2}{2} f_t - \frac{1}{2} S^{(i)}_t \right) \dv^g \\
&\qquad=
\int_M \left(
D_g\scal(k^{(i+1)}_t) + D^2_g\scal(h^{(i+1)}_t,h^{(i+1)}_t) \right) \dv^g \\
&\qquad=
\int_M \left(
(\Delta - \sigma) \left( \trace^g k^{(i+1)}_t - |h^{(i+1)}_t|^2_g \right)
+ \delta(\delta k^{(i+1)}_t + \delta' h^{(i+1)}_t)
- \frac{1}{2}\langle \Delta_E h^{(i+1)}_t,h^{(i+1)}_t\rangle
\right) \dv^g \\
&\qquad=
\int_M \left( \sigma V^{(i)}_t
- \frac{1}{2}\langle \Delta_E h^{(i+1)}_t,h^{(i+1)}_t\rangle
\right) \dv^g,
\end{split} \]
and using \eqref{hpm} we have
\begin{equation} \label{def-lambdat-2} 
\int_M \left( \frac{t^2}{2} f_t - \frac{1}{2} S^{(i)}_t 
- \sigma V^{(i)}_t \right) \dv^g \\
= \begin{cases} 
F_0 t^2(1+\lambda_t^{(i+1)})^2, &\text{if } F_0 \neq 0 ,\\
-t^2 \lambda_t^{(i+1)}, &\text{if } F_0 = 0.
\end{cases}
\end{equation}
We see that \eqref{def-lambdat-2} defines $\lambda_t^{(i+1)}$ and hence $h_t^{(i+1)}$. We determine then $\trace^g k^{(i+1)}_t$ from \eqref{prescribe-tr-k-i-2}.

Finally, with 
$k^{(i)}_t = \frac{1}{n}(\trace^g k^{(i)}_t) g + \mathring{k}^{(i)}_t$ 
where $\trace^g \mathring{k}^{(i)}_t = 0$, Equation \eqref{iteration-h-k-2} becomes
\[ \begin{split}
&(\Delta - \sigma) \left(\trace^g k^{(i+1)}_t - |h^{(i+1)}_t|^2_g \right)
+ \delta \left(
\frac{1}{n}d\trace^g k^{(i+1)}_t + \delta \mathring{k}^{(i+1)}_t
+ \delta' h^{(i+1)}_t \right) \\
&\qquad- \frac{1}{2}\langle \Delta_E h^{(i+1)}_t,h^{(i+1)}_t\rangle 
= \frac{t^2}{2}f_t - \frac{1}{2} S^{(i)}_t
\end{split}\]
or with \eqref{prescribe-tr-k-i-2}, 
\[
\begin{split}
&\delta \left( \delta \mathring{k}^{(i+1)}_t \right) = 
\frac{t^2}{2}f_t - \frac{1}{2} S^{(i)}_t + (\Delta - \sigma) V^{(i)}_t \\
&\qquad
+ \frac{1}{2}\langle \Delta_E h^{(i+1)}_t,h^{(i+1)}_t\rangle
-\delta \left( \frac{1}{n}d\trace^g k^{(i+1)}_t + \delta' h^{(i+1)}_t \right).
\end{split}
\]
Therefore,
\begin{equation} \label{def-ringkt-2} 
\begin{split}
\mathring{k}^{(i+1)}_t 
&= Q\Bigg(  
\frac{t^2}{2}f_t - \frac{1}{2} S^{(i)}_t + (\Delta - \sigma) V^{(i)}_t \\
&\qquad \qquad
+ \frac{1}{2}\langle \Delta_E h^{(i+1)}_t,h^{(i+1)}_t\rangle
-\delta \left( \frac{1}{n}d\trace^g k^{(i+1)}_t + \delta' h^{(i+1)}_t \right) \Bigg).
\end{split}
\end{equation}
Using \eqref{prescribe-tr-k-i-2}, \eqref{def-lambdat-2}, \eqref{def-ringkt-2}  we will continue the estimates for a contraction.

{\bf Step 5: Estimates for convergence.}
We will now find estimates for the functions $S^{(i)}_t$ and $V^{(i)}_t$.
We set $w^{(i)} := sh^{(i)}_t + \frac{s^2}{2}k^{(i)}_t$ so that 
$g^{(i)}_{t,s} = g + w^{(i)}$. 
Further, set $\tilde{g}^{(i+1)}_r := g + rw^{(i+1)} + (1-r) w^{(i)}$.
We have
\[ \begin{split}
\scal(g^{(i+1)}_{t,s}) - \scal(g^{(i)}_{t,s}) 
&= \scal(g + w^{(i+1)}) - \scal(g + w^{(i)}) \\
&= \int_0^1 \frac{d}{dr} \scal( \tilde{g}^{(i+1)}_r ) \, dr \\
&= \int_0^1 D_{\tilde{g}^{(i+1)}_r }\scal( w^{(i+1)} - w^{(i)} ) \, dr .
\end{split} \]
This gives us 
\begin{equation} \label{diffofd3scal}
\begin{split}
&\frac{d^3}{ds^3}\left(\scal(g^{(i+1)}_{t,s}) - \scal(g^{(i)}_{t,s}) \right) \\
&\qquad
= 
\frac{d^3}{ds^3} \int_0^1 D_{\tilde{g}^{(i+1)}_r}\scal( w^{(i+1)} - w^{(i)} ) \, dr \\
&\qquad
= 
\int_0^1 \frac{d^3}{ds^3}\left(D_{\tilde{g}^{(i+1)}_r}\scal\right)( w^{(i+1)} - w^{(i)} ) \, dr \\
&\qquad\qquad
+ 3\int_0^1 \frac{d^2}{ds^2}\left(D_{\tilde{g}^{(i+1)}_r}\scal\right)
( \frac{d}{ds}(w^{(i+1)} - w^{(i)}) ) \, dr \\
&\qquad\qquad
+ 3\int_0^1 \frac{d}{ds}\left(D_{\tilde{g}^{(i+1)}_r}\scal\right)
( \frac{d^2}{ds^2}(w^{(i+1)} - w^{(i)}) ) \, dr \\
&\qquad
= 
\int_0^1 \frac{d^3}{ds^3}\left(D_{\tilde{g}^{(i+1)}_r}\scal\right)
\left(
s(h^{(i+1)}_t - h^{(i)}_t) + \frac{s^2}{2}(k^{(i+1)}_t - k^{(i)}_t) 
\right) \, dr \\
&\qquad\qquad
+ 3\int_0^1 \frac{d^2}{ds^2}\left(D_{\tilde{g}^{(i+1)}_r}\scal\right)
\left( (h^{(i+1)}_t - h^{(i)}_t) + s(k^{(i+1)}_t - k^{(i)}_t) \right) 
\, dr \\
&\qquad\qquad
+ 3\int_0^1 \frac{d}{ds}\left(D_{\tilde{g}^{(i+1)}_r}\scal\right)
\left( k^{(i+1)}_t - k^{(i)}_t \right) \, dr. \\
\end{split} 
\end{equation}
For the three terms here we have
\[ \begin{split}
&\frac{d^3}{ds^3}\left(D_{\tilde{g}^{(i+1)}_r}\scal\right)
\left(
s(h^{(i+1)}_t - h^{(i)}_t) + \frac{s^2}{2}(k^{(i+1)}_t - k^{(i)}_t) 
\right) \\ 
&\qquad
=
\frac{d^2}{ds^2}\left(D^2_{\tilde{g}^{(i+1)}_r}\scal\right)
\left(
\frac{d}{ds} \tilde{g}^{(i+1)}_r,
s(h^{(i+1)}_t - h^{(i)}_t) + \frac{s^2}{2}(k^{(i+1)}_t - k^{(i)}_t) 
\right) \\ 
&\qquad
=
\frac{d^2}{ds^2}\left(D^2_{\tilde{g}^{(i+1)}_r}\scal\right)
\left(
r(h^{(i+1)}_t + sk^{(i+1)}_t) + (1-r)(h^{(i)}_t + s k^{(i)}_t),
s(h^{(i+1)}_t - h^{(i)}_t) + \frac{s^2}{2}(k^{(i+1)}_t - k^{(i)}_t) 
\right) 
\end{split} \]
and
\[ \begin{split}
&\frac{d^2}{ds^2}\left(D_{\tilde{g}^{(i+1)}_r}\scal\right)
\left(
(h^{(i+1)}_t - h^{(i)}_t) + s(k^{(i+1)}_t - k^{(i)}_t) 
\right) \\ 
&\qquad
=
\frac{d}{ds}\left(D^2_{\tilde{g}^{(i+1)}_r}\scal\right)
\left(
\frac{d}{ds} \tilde{g}^{(i+1)}_r,
(h^{(i+1)}_t - h^{(i)}_t) + s(k^{(i+1)}_t - k^{(i)}_t) 
\right) \\ 
&\qquad
=
\frac{d}{ds}\left(D^2_{\tilde{g}^{(i+1)}_r}\scal\right)
\left(
r(h^{(i+1)}_t + sk^{(i+1)}_t) + (1-r)(h^{(i)}_t + s k^{(i)}_t),
(h^{(i+1)}_t - h^{(i)}_t) + s(k^{(i+1)}_t - k^{(i)}_t) 
\right) 
\end{split} \]
and
\[ \begin{split}
&\frac{d}{ds}\left(D_{\tilde{g}^{(i+1)}_r}\scal\right)
\left(
k^{(i+1)}_t - k^{(i)}_t
\right) \\ 
&\qquad
=
D^2_{\tilde{g}^{(i+1)}_r}\scal
\left(\frac{d}{ds} \tilde{g}^{(i+1)}_r,
k^{(i+1)}_t - k^{(i)}_t
\right) \\ 
&\qquad
=
D^2_{\tilde{g}^{(i+1)}_r}\scal
\left(
r(h^{(i+1)}_t + sk^{(i+1)}_t) + (1-r)(h^{(i)}_t + s k^{(i)}_t),
k^{(i+1)}_t - k^{(i)}_t 
\right) .
\end{split} \]
The operator $D^2_{\tilde{g}^{(i+1)}_r}\scal$ has the schematic form 
\[ \begin{split}
D^2_{\tilde{g}^{(i+1)}_r}\scal(h_1,h_2) 
&=  
\nabla^2 h_1 * h_2
+ \nabla h_1 * \nabla h_2
+ h_1 * \nabla^2 h_2 \\
&\qquad
+ S * \nabla h_1 * h_2
+ T * h_1 * \nabla h_2
+ R * h_1 * h_2 ,
\end{split}\]
see for example Lemma A.3 in \cite{Kroenckearxiv2013}. Here the ``$*$'' denotes product of tensors in coordinates, followed by a combination of index raising, (anti-)symmetrizing, and contractions. The expressions $S,T$ involve the metric $\tilde{g}^{(i+1)}_r$ and its first derivatives. The expression $R$ involves the metric $\tilde{g}^{(i+1)}_r$ and its first and second derivatives. The first order terms come from the fact that we change from the covariant derivative of $\tilde{g}^{(i+1)}_r$ to the covariant derivative $\nabla$ of $g$. The operators 
$\frac{d}{ds}\left(D^2_{\tilde{g}^{(i+1)}_r}\scal\right)$ and 
$\frac{d^2}{ds^2}\left(D^2_{\tilde{g}^{(i+1)}_r}\scal\right)$ have similar schematic forms. From \eqref{diffofd3scal} we thus find from standard estimates that
\begin{equation} \label{estimate-Si-int}
\begin{split}
& \left| \int_M \left( S^{(i+1)}_t - S^{(i)}_t \right) \dv^g \right| \\
& \leq 
\int_M \int_0^1 (1-s)^2
\left| \frac{d^3}{ds^3} \left(
\scal(g^{(i+1)}_{t,s}) - \scal(g^{(i)}_{t,s})
\right) \right|
\, ds \dv^g \\
& \leq
C\left(\left\|h^{(i)}_t\right\|_{\psi^2\phi^2 H^{k+2}_{\phi,\psi}}+\left\|h^{(i+1)}_t\right\|_{\psi^2\phi^2 H^{k+2}_{\phi,\psi}}+\left\|k^{(i)}_t\right\|_{\psi^2\phi^2 H^{k+2}_{\phi,\psi}}+\left\|k^{(i+1)}_t\right\|_{\psi^2\phi^2 H^{k+2}_{\phi,\psi}}\right)\cdot\\
&
\qquad\left(
\|(h^{(i+1)}_t - h^{(i)}_t)\|_{\psi^2\phi^2 H^{k+2}_{\phi,\psi}} 
+ \| k^{(i+1)}_t - k^{(i)}_t \|_{\psi^2\phi^2 H^{k+2}_{\phi,\psi}}
\right),
\end{split} 
\end{equation}
as long as the $h^{(i)}_t, h^{(i+1)}_t, k^{(i)}_t, k^{(i+1)}_t$ are bounded in the $\psi^2\phi^2 H^{k+2}_{\phi,\psi}$ norm.
From \eqref{diffofd3scal} together with Theorem \ref{thmproductsofderivatives} we have 
\begin{equation} \label{estimate-Si-norm}
\begin{split}
&\left\| 
S^{(i+1)}_t - S^{(i)}_t
\right\|_{\psi^2 H^{k}_{\phi,\psi}} \\
& \leq 
C\left(\left\|h^{(i)}_t\right\|_{\psi^2\phi^2 H^{k+2}_{\phi,\psi}}+\left\|h^{(i+1)}_t\right\|_{\psi^2\phi^2 H^{k+2}_{\phi,\psi}}+\left\|k^{(i)}_t\right\|_{\psi^2\phi^2 H^{k+2}_{\phi,\psi}}+\left\|k^{(i+1)}_t\right\|_{\psi^2\phi^2 H^{k+2}_{\phi,\psi}}\right)\cdot
\\& \qquad\left(
\|(h^{(i+1)}_t - h^{(i)}_t)\|_{\psi^2\phi^2 H^{k+2}_{\phi,\psi}}
+ \| k^{(i+1)}_t - k^{(i)}_t \|_{\psi^2\phi^2 H^{k+2}_{\phi,\psi}}
\right).
\end{split}
\end{equation}
Next, we have 
\[
\left( V^{(i+1)}_t - V^{(i)}_t \right)  dV(g) 
= \int_0^1 (1-s)^2 \frac{d^3}{ds^3} 
\left(
dV({ g^{(i+1)}_{t,s} }) - dV({ g^{(i)}_{t,s} })
\right) \, ds,
\]
and we get similar estimate as above, but with no decrease in derivatives.
First, 
\begin{equation} \label{estimate-Vi-int}
\begin{split}
& \left| \int_M \left( V^{(i+1)}_t - V^{(i)}_t \right) \dv^g \right| \\
&\leq C\left(\left\|h^{(i)}_t\right\|_{\psi^2\phi^2 H^{k}_{\phi,\psi}}+\left\|h^{(i+1)}_t\right\|_{\psi^2\phi^2 H^{k}_{\phi,\psi}}+\left\|k^{(i)}_t\right\|_{\psi^2\phi^2 H^{k}_{\phi,\psi}}+\left\|k^{(i+1)}_t\right\|_{\psi^2\phi^2 H^{k}_{\phi,\psi}}\right)\cdot\\ &\qquad
\left(
\|(h^{(i+1)}_t - h^{(i)}_t)\|_{\psi^2\phi^2 H^{k}_{\phi,\psi}} 
+ \| k^{(i+1)}_t - k^{(i)}_t \|_{\psi^2\phi^2 H^{k}_{\phi,\psi}}
\right),
\end{split} 
\end{equation}
and second
\begin{equation} \label{estimate-Vi-norm}
\begin{split}
&\left\| 
V^{(i+1)}_t - V^{(i)}_t
\right\|_{\psi^2 H^{k}_{\phi,\psi}} \\
& \leq 
C\left(\left\|h^{(i)}_t\right\|_{\psi^2\phi^2 H^{k+2}_{\phi,\psi}}+\left\|h^{(i+1)}_t\right\|_{\psi^2\phi^2 H^{k+2}_{\phi,\psi}}+\left\|k^{(i)}_t\right\|_{\psi^2\phi^2 H^{k+2}_{\phi,\psi}}+\left\|k^{(i+1)}_t\right\|_{\psi^2\phi^2 H^{k+2}_{\phi,\psi}}\right)\cdot\\& \qquad\left(
\|(h^{(i+1)}_t - h^{(i)}_t)\|_{\psi^2\phi^2 H^{k}_{\phi,\psi}}
+ \| k^{(i+1)}_t - k^{(i)}_t \|_{\psi^2\phi^2 H^{k}_{\phi,\psi}}
\right).
\end{split}
\end{equation}

{\bf Step 6: Convergence.}
From \eqref{def-lambdat-2} we get
\[ \begin{split}
&- \int_N \left(
\frac{1}{2} \left(S^{(i+1)}_t - S^{(i)}_t \right)
+ \sigma \left( V^{(i+1)}_t -  V^{(i)}_t \right) 
\right) \dv^g \\
&\qquad=
\begin{cases} 
F_0 t^2 \left( (1+\lambda_t^{(i+2)})^2 - (1+\lambda_t^{(i+1)})^2 \right),
&\text{if } F_0 \neq 0 ,\\
t^2 \left( \lambda_t^{(i+2)} - \lambda_t^{(i+1)} \right),
&\text{if } F_0 = 0.
\end{cases}
\end{split} \]
Thus by \eqref{estimate-Si-int} and  \eqref{estimate-Vi-int} we find 
\[ \begin{split}
& t^2 | \lambda_t^{(i+2)} - \lambda_t^{(i+1)}| \\
&\qquad \leq C
\left(
\|(h^{(i+1)}_t - h^{(i)}_t)\|_{\psi^2\phi^2 H^{k+2}_{\phi,\psi}} 
+ \| k^{(i+1)}_t - k^{(i)}_t \|_{\psi^2\phi^2 H^{k+2}_{\phi,\psi}}
\right).
\end{split} \]
Therefore,
\begin{equation} \label{contraction-h}
\begin{split}
&\left\| h^{(i+2)}_t - h^{(i+1)}_t \right\|^2_{\psi^2\phi^2 H^{k+2}_{\phi,\psi}} \\
& \leq 
C t^2 \left| \lambda^{(i+2)}_t - \lambda^{(i+1)}_t \right| \\
& \leq C\left(\left\|h^{(i)}_t\right\|_{\psi^2\phi^2 H^{k+2}_{\phi,\psi}}+\left\|h^{(i+1)}_t\right\|_{\psi^2\phi^2 H^{k+2}_{\phi,\psi}}+\left\|k^{(i)}_t\right\|_{\psi^2\phi^2 H^{k+2}_{\phi,\psi}}+\left\|k^{(i+1)}_t\right\|_{\psi^2\phi^2 H^{k+2}_{\phi,\psi}}\right)\cdot\\
&\qquad\left(
\|(h^{(i+1)}_t - h^{(i)}_t)\|^2_{\psi^2\phi^2 H^{k+2}_{\phi,\psi}} 
+ \| k^{(i+1)}_t - k^{(i)}_t \|^2_{\psi^2\phi^2 H^{k+2}_{\phi,\psi}}
\right),
\end{split}
\end{equation}
which is valid as long as the $h^{(i)}_t, k^{(i)}_t$ are uniformly bounded in $\psi^2\phi^2 \mathring{H}^{k+2}_{\phi,\psi}$.
Note that the first estimate of the $\psi^2\phi^2 H^{k+2}_{\phi,\psi}$ norm follows since the $h^{(i)}_t$ are in a 1-dimensional family.

Next, we find a contraction bound for $\trace^g k^{(i)}_t$. From \eqref{prescribe-tr-k-i-2} we have
\[ \begin{split}
&\trace^g k^{(i+2)}_t - \trace^g k^{(i+1)}_t \\
&\qquad
= |h^{(i+2)}_t|^2_g - |h^{(i+1)}_t|^2_g
- \left( V^{(i+1)}_t - V^{(i)}_t \right) \\
&\qquad
=\left\langle h^{(i+2)}_t + h^{(i+1)}_t, h^{(i+2)}_t - h^{(i+1)}_t \right\rangle
- \left( V^{(i+1)}_t - V^{(i)}_t \right), \\
\end{split}\]
so by \eqref{contraction-h} and \eqref{estimate-Vi-norm} we have 
\begin{equation} \label{contraction-trk}
\begin{split}
&\left\| 
\trace^g k^{(i+2)}_t - \trace^g k^{(i+1)}_t
\right\|_{\psi^2\phi^2 H^{k+2}_{\phi,\psi}} \\
& \leq C\left(\left\|h^{(i)}_t\right\|_{\psi^2\phi^2 H^{k+2}_{\phi,\psi}}+\left\|h^{(i+1)}_t\right\|_{\psi^2\phi^2 H^{k+2}_{\phi,\psi}}+\left\|k^{(i)}_t\right\|_{\psi^2\phi^2 H^{k+2}_{\phi,\psi}}+\left\|k^{(i+1)}_t\right\|_{\psi^2\phi^2 H^{k+2}_{\phi,\psi}}\right)\cdot\\&
\qquad\left(
\|(h^{(i+1)}_t - h^{(i)}_t)\|_{\psi^2\phi^2 H^{k+2}_{\phi,\psi}} 
+ \| k^{(i+1)}_t - k^{(i)}_t \|_{\psi^2\phi^2 H^{k+2}_{\phi,\psi}}
\right).
\end{split}
\end{equation}
Finally, we prove a contraction bound for $\mathring{k}^{(i)}_t$. 
For this, we write \eqref {def-ringkt-2} as
\[
\mathring{k}^{(i+1)}_t := Q( f^{(i+1)}_t ),
\]
where
\[ \begin{split}
f^{(i+1)}_t &:=
\frac{t^2}{2}f_t - \frac{1}{2} S^{(i)}_t + (\Delta - \sigma) V^{(i)}_t \\
&\qquad
+ \frac{1}{2}\langle \Delta_E h^{(i+1)}_t,h^{(i+1)}_t\rangle
-\delta \left( \frac{1}{n}d\trace^g k^{(i+1)}_t + \delta' h^{(i+1)}_t \right).
\end{split}\]
Thus
$\mathring{k}^{(i+2)}_t - \mathring{k}^{(i+1)}_t 
= Q( f^{(i+2)}_t - f^{(i+1)}_t )$,
where
\[\begin{split}
f^{(i+2)}_t - f^{(i+1)}_t &:=
- \frac{1}{2} \left(S^{(i+1)}_t - S^{(i)}_t\right) 
+ (\Delta - \sigma) \left(V^{(i+1)}_t - V^{(i)}_t\right) \\
&\qquad
+ \frac{1}{2} \left( 
\langle \Delta_E h^{(i+2)}_t,h^{(i+2)}_t\rangle
-\langle \Delta_E h^{(i+1)}_t,h^{(i+1)}_t\rangle \right) \\
&\qquad
-\delta \left( 
\frac{1}{n}d\left( \trace^g k^{(i+2)}_t - \trace^g k^{(i+1)}_t \right)
+ \delta' h^{(i+2)}_t -  \delta' h^{(i+1)}_t
\right).
\end{split}\]
To apply the estimate \eqref{estimate-delay}, we need to estimate the 
$\psi^2 H^{k}_{\phi,\psi}$-norm of the above. 
The first two terms on the left hand side can be estimated using \eqref{estimate-Si-norm} and \eqref{estimate-Vi-norm} and the remaining terms can be estimated using \eqref{contraction-h} and \eqref{contraction-trk}. 
We thus get
\begin{equation} \label{contraction-ringk}
\begin{split}
&\| \mathring{k}^{(i+2)}_t - \mathring{k}^{(i+1)}_t \|_{\psi^2\phi^2 H^{k+2}_{\phi,\psi}} \\
& \leq C
\left(\left\|h^{(i)}_t\right\|_{\psi^2\phi^2 H^{k+2}_{\phi,\psi}}+\left\|h^{(i+1)}_t\right\|_{\psi^2\phi^2 H^{k+2}_{\phi,\psi}}+\left\|k^{(i)}_t\right\|_{\psi^2\phi^2 H^{k+2}_{\phi,\psi}}+\left\|k^{(i+1)}_t\right\|_{\psi^2\phi^2 H^{k+2}_{\phi,\psi}}\right)\cdot\\ & 
\qquad\left(
\|(h^{(i+1)}_t - h^{(i)}_t)\|_{\psi^2\phi^2 H^{k+2}_{\phi,\psi}} 
+ \| k^{(i+1)}_t - k^{(i)}_t \|_{\psi^2\phi^2 H^{k+2}_{\phi,\psi}}
\right).
\end{split}
\end{equation}
Finally, there exists an $\epsilon>0$ such that if 
\begin{equation}\label{eq:small_norms}
\left\|h^{(i)}_t\right\|_{\psi^2\phi^2 H^{k+2}_{\phi,\psi}}+\left\|h^{(i+1)}_t\right\|_{\psi^2\phi^2 H^{k+2}_{\phi,\psi}}+\left\|k^{(i)}_t\right\|_{\psi^2\phi^2 H^{k+2}_{\phi,\psi}}+\left\|k^{(i+1)}_t\right\|_{\psi^2\phi^2 H^{k+2}_{\phi,\psi}}<\epsilon,
\end{equation}
we get from \eqref{contraction-h},
\eqref{contraction-trk} and \eqref{contraction-ringk} the contraction property
\[ \begin{split}
&\|(h^{(i+2)}_t - h^{(i+1)}_t)\|_{\psi^2\phi^2 H^{k+2}_{\phi,\psi}} 
+ \| k^{(i+2)}_t - k^{(i+1)}_t \|_{\psi^2\phi^2 H^{k+2}_{\phi,\psi}} \\
&\qquad \leq \frac{1}{3}
\left(
\|(h^{(i+1)}_t - h^{(i)}_t)\|_{\psi^2\phi^2 H^{k+2}_{\phi,\psi}} 
+ \| k^{(i+1)}_t - k^{(i)}_t \|_{\psi^2\phi^2 H^{k+2}_{\phi,\psi}}
\right).
\end{split}\]
By a  standard induction argument we show that \eqref{eq:small_norms} holds for all $i$.
By the first step of the proof, we can choose $t$ so small that 
\begin{equation}\label{eq:small_norms_2}
\left\|h^{(0)}_t\right\|_{\psi^2\phi^2 H^{k+2}_{\phi,\psi}}+\left\|h^{(1)}_t\right\|_{\psi^2\phi^2 H^{k+2}_{\phi,\psi}}+\left\|k^{(0)}_t\right\|_{\psi^2\phi^2 H^{k+2}_{\phi,\psi}}+\left\|k^{(1)}_t\right\|_{\psi^2\phi^2 H^{k+2}_{\phi,\psi}}<\frac{\epsilon}{5},
\end{equation}
and \eqref{eq:small_norms} is shown for $i=0$. Suppose that it holds for $i$, then the triangle inequality, the contraction property and \eqref{eq:small_norms_2} yield
\[ \begin{split}
&\left\|h^{(i+1)}_t\right\|_{\psi^2\phi^2 H^{k+2}_{\phi,\psi}}+\left\|h^{(i+2)}_t\right\|_{\psi^2\phi^2 H^{k+2}_{\phi,\psi}}+\left\|k^{(i+1)}_t\right\|_{\psi^2\phi^2 H^{k+2}_{\phi,\psi}}+\left\|k^{(i+2)}_t\right\|_{\psi^2\phi^2 H^{k+2}_{\phi,\psi}}\\
&\qquad\leq
\left\|h^{(0)}_t\right\|_{\psi^2\phi^2 H^{k+2}_{\phi,\psi}}+\left\|h^{(1)}_t\right\|_{\psi^2\phi^2 H^{k+2}_{\phi,\psi}}+\left\|k^{(0)}_t\right\|_{\psi^2\phi^2 H^{k+2}_{\phi,\psi}}+\left\|k^{(1)}_t\right\|_{\psi^2\phi^2 H^{k+2}_{\phi,\psi}}\\
&\qquad\qquad 
+\sum_{j=0}^{i}\left(
\|(h^{(j+1)}_t - h^{(j)}_t)\|_{\psi^2\phi^2 H^{k+2}_{\phi,\psi}} 
+ \| k^{(j+1)}_t - k^{(j)}_t \|_{\psi^2\phi^2 H^{k+2}_{\phi,\psi}}\right) \\
&\qquad \qquad 
+\sum_{j=0}^{i}\left(
\|(h^{(j+2)}_t - h^{(j+1)}_t)\|_{\psi^2\phi^2 H^{k+2}_{\phi,\psi}} 
+ \| k^{(j+2)}_t - k^{(j+1)}_t \|_{\psi^2\phi^2 H^{k+2}_{\phi,\psi}} \right)\\
&\qquad 
\leq \left\|h^{(0)}_t\right\|_{\psi^2\phi^2 H^{k+2}_{\phi,\psi}}+\left\|h^{(1)}_t\right\|_{\psi^2\phi^2 H^{k+2}_{\phi,\psi}}+\left\|k^{(0)}_t\right\|_{\psi^2\phi^2 H^{k+2}_{\phi,\psi}}+\left\|k^{(1)}_t\right\|_{\psi^2\phi^2 H^{k+2}_{\phi,\psi}}\\
&\qquad\qquad
+\sum_{j=0}^{i}\left(\frac{1}{3^{j}}+\frac{1}{3^{j+1}}\right)
\left(
\|(h^{(1)}_t - h^{(0)}_t)\|_{\psi^2\phi^2 H^{k+2}_{\phi,\psi}} 
+ \| k^{(1)}_t - k^{(0)}_t \|_{\psi^2\phi^2 H^{k+2}_{\phi,\psi}}
\right)\\
&\qquad
< \frac{\epsilon}{5}+\sum_{j=0}^{i}\left(\frac{1}{3^{j}}+\frac{1}{3^{j+1}}\right)\cdot\frac{2\epsilon}{5}
<\epsilon,
\end{split} \]
so \eqref{eq:small_norms} does also hold for $i+1$.
The contraction property implies that the sequences
$h^{(i)}_t, k^{(i)}_t$ converge in $\psi^2\phi^2 H^{k+2}_{\phi,\psi}$ to limits $h^{(\infty)}_t, k^{(\infty)}_t$. 
From \eqref{iteration-h-k} and \eqref{iteration-dV} we see that the metric $g^{(\infty)}_t := g + h^{(\infty)}_t + \frac{1}{2}k^{(\infty)}_t$ satisfies \eqref{scal=f_t} and \eqref{dV-preserved}. It is clear from the construction that $g^{(\infty)}_0=g$, $\frac{d}{dt} g^{(\infty)}_t |_{t=0} = h$, and $g^{(\infty)}_t = g$ outside of $\Omega$.

{\bf Step 7: Regularity.}
The last step of the proof is to show that the metric $g^{(\infty)}_t$ is smooth. By construction we have that $h^{(\infty)}_t$ is smooth. 
From \eqref{dV-preserved} we have 
\[
\dv\left(g + h^{(\infty)}_t + \frac{1}{2n} (\trace^g k^{(\infty)}_t) g
+ \frac{1}{2}\mathring{k}^{(\infty)}_t \right)
=\dv(g).
\]
Let $\{e_i\}$ be an oriented orthonormal frame for the metric $g$ on an open set $U$. Then the matrix-valued function 
\[
\gamma_t := g^{(\infty)}_t (e_i,e_j)
= \left( 1 +  \frac{1}{2n} \trace^g k^{(\infty)}_t \right) \delta_{ij}
+ h^{(\infty)}_t(e_i,e_j)
+ \frac{1}{2}\mathring{k}^{(\infty)}_t(e_i,e_j)
\]
satisfies $\mathrm{det}(\gamma_t) = 1$, so it gives a curve of maps $\gamma_t: U \to SL(n,\R)$. Note that the second and the third term in $\gamma_t$ are both trace-free. Let $\Pi: SL(n,\R) \to \mathfrak{sl}(n,\R)$ be the projection of a matrix on its trace-free part, that is $\Pi(A) = A - \frac{1}{n} \trace(A) I$. 
Since $\mathfrak{sl}(n,\R) = T_I SL(n,\R)$ and $D_I\Pi=\operatorname{Id}$, there is a neighbourhood $V_I$ of $I \in SL(n,\R)$ and a neighbourhood $V_0$ of $0 \in \mathfrak{sl}(n,\R)$ such that $\Pi: V_I \to V_0$ is a diffeomorphism. Since the curve $\gamma_t$ has $\gamma_0 = I$ we thus find that if 
$h^{(\infty)}_t + \frac{1}{2}\mathring{k}^{(\infty)}_t$
is small enough we have
\[
\gamma_t 
= \Pi^{-1} \left( h^{(\infty)}_t(e_i,e_j) + \frac{1}{2}\mathring{k}^{(\infty)}_t(e_i,e_j) \right)
\]
and
\[
n + \frac{1}{2} \trace^g k^{(\infty)}_t 
= \trace (\gamma_t) 
= \trace \left(\Pi^{-1} \left( 
h^{(\infty)}_t(e_i,e_j) + \frac{1}{2}\mathring{k}^{(\infty)}_t(e_i,e_j) \right) \right).
\]
We conclude that $\trace^g k^{(\infty)}_t$ has at least the same order of regularity as $\mathring{k}^{(\infty)}_t$.  
From \eqref{def-ringkt-2} we have
\[\begin{split}
\mathring{k}^{(\infty)}_t 
&= Q\Bigg(  
\frac{t^2}{2}f_t - \frac{1}{2} S^{(\infty)}_t + (\Delta - \sigma) V^{(\infty)}_t \\
&\qquad \qquad
+ \frac{1}{2}\langle \Delta_E h^{(\infty)}_t,h^{(\infty)}_t\rangle
-\delta \left( \frac{1}{n}d\trace^g k^{(\infty)}_t + \delta' h^{(\infty)}_t \right) \Bigg),
\end{split}\]
so by \eqref{estimate-delay}, the regularity of $\mathring{k}^{(\infty)}_t$ is two orders higher than the argument of $Q$ in the left hand side. By bootstrapping, we see that $\trace^g k^{(\infty)}_t$ and $\mathring{k}^{(\infty)}_t$ are both smooth.

This finishes the proof of Theorem \ref{thm-prescribe-scal}.
\end{proof}
\begin{rem}\label{rem:prescribing_scal_Ricci_flat}
If $(M,g)$ is Ricci-flat and we drop the assumption of preserving the volume element, we can weaken the warped product assumption in Theorem \ref{thm-prescribe-scal}. In this case, we may not allow $(M,g)$ to be locally a pure product but it can be a Ricci-flat cone. The reason is the extension of the domain of definition of $P=\delta^2$ from trace-free symmetric two-tensors to all symmetric two-tensors. Then the formal adjoint $P^*$ of $P$ is the Hessian $\nabla^2$ and not its trace-free part $\mathring{\nabla}^2$ and the assumption $\mathrm{ker}(P^*)=\left\{0\right\}$ leads to weaker geometric conclusions.
\end{rem}

\section{A generalized $\lambda$-functional}
\label{sec_lambda}

Let $(M,\hat{g})$ be a compact Riemannian manifold with smooth (but possibly empty) boundary 
and let $\mathcal{M}^{\infty}_{\hat{g}}$ be the set of smooth metrics on $M$ such that $g - \hat{g}$ vanishes to every order at $\partial M$. Let $C^{\infty}(M)$ be the set of smooth functions on $M$. For $\alpha>0$, define
\[
F_{\alpha}:\mathcal{M}^{\infty}_{\hat{g}}\times C^{\infty}(M)\to \R, \qquad 
F_{\alpha}(g,f) := \int_M \left( \scal+\alpha|\nabla f|^2 \right)e^{-f} \dv ,
\]
and
\[
\lambda_{\alpha}(g)
:= \inf\left\{F_{\alpha}(g,f) \mid f\in C^{\infty}(M), \int_M e^{-f}\dv=1 \right\}.
\]
For closed manifolds and with $\alpha=1$, this is the $\lambda$-functional introduced by Perelman \cite{perelman2002entropy}. The parameter-dependent version has been used in different contexts, see for example \cites{BD03,LM21}. 

The substitution $\omega=e^{-\frac{f}{2}}$ shows that
\[
\lambda_{\alpha}(g)
= \inf\left\{G_{\alpha}(g,\omega)\mid \omega\in C^{\infty}_{+}(M), \int_M \omega^2\dv=1 \right\},
\]
where
\[
G(g,\omega) := \int_M (4{\alpha}|\nabla \omega|^2 + \scal \omega^2)\dv.
\]
By standard theory, $\lambda_{\alpha}(g)$ is the smallest (Neumann) eigenvalue of the Schr\"{o}dinger operator $4{\alpha}\Delta+\scal$. Moreover, the minimizing function $\omega_g$ is the unique eigenfunction of constant positive sign satisfying the normalization condition $\int_M \omega^2 \dv=1$. Therefore, the minimizer $f_g=-2\log(\omega_g)$ of $F_{\alpha}(g,f)$ satisfies
\begin{equation} \label{eq:eulerlagrange_generallambda}
-2\alpha\Delta f_g-\alpha|\nabla f_g|^2+\scal
=\lambda_{\alpha}(g),\qquad \nabla_{\nu}f_g=0.
\end{equation}

\begin{rem}
If one works on manifolds with boundary, one could either consider the smallest Dirichlet or the smallest Neumann eigenvalue of $4{\alpha}\Delta+\scal$. It turns out that if we consider the smallest Neumann eigenvalue, the variational theory of the functional is much simpler and in fact almost parallel to the case of closed manifolds.
\end{rem}

The following lemma is elementary, but we use it later in an essential way.
\begin{lem} \label{lem:lambda_and_rigidity}
Let $(M,g)$ be a compact manifold with or without boundary and let $\alpha>0$. If $\scal \geq c$ for some constant $c\in\R$, we have $\lambda_{\alpha}(g)\geq c$. Moreover, if $\scal\not\equiv c$, then $ \lambda_{\alpha}(g)> c$.
\end{lem}

\begin{proof}
Let $\omega$ be the positive eigenfunction of the operator $4\alpha\Delta+\scal$ with eigenvalue $\lambda_{\alpha}(g)$ and the normalization condition $\int_M \omega^2\dv=1$. Then
\[
\lambda_{\alpha}(g) =\int_M(4\alpha|\nabla\omega|^2 + \scal \omega^2)\dv
\geq \int_M \scal \omega^2 \dv
\geq c.
\]
Suppose now that $\scal \not\equiv c$. We may assume that $c=\min_M\scal$. Then $\scal$ is nonconstant so $\omega$ has to be nonconstant as well. Consequently, the function $|\nabla \omega|^2$ does not vanish identically and the first of the two inequalities above is strict. This proves the lemma.
\end{proof}
Throughout the following sections, we write $f$ instead of $f_g$
for the minimizer $f_g$ in the definition of $\lambda_{\alpha}(g)$,
 whenever it is clear from the context which metric we consider.

\subsection{The first variation formula}

\begin{prop}\label{prop_first_variation_general_lambda}
The first variation of $\lambda_{\alpha}(g)$ is given by
\[ \begin{split}
D_g\lambda_{\alpha}(h)
&=-\int_M \langle \ric + \nabla^2f-(\alpha-1)\nabla f\otimes\nabla f,h\rangle e^{-f} \dv\\
&\qquad +\int_M \langle \frac{1}{2}(1-\frac{1}{\alpha})(\scal-\lambda_{\alpha}(g))g+\frac{1}{2}(\alpha-1)|\nabla f|^2g,h\rangle e^{-f}\dv.
\end{split} \]
This formula implies three assertions:
\begin{itemize}
\item[(i)]  Constant scalar curvature metrics are critical points with respect to volume-preserving conformal deformations.
\item[(ii)] Einstein metrics are critical points with respect to volume-preserving deformations.
\item[(iii)] Ricci-flat metrics are critical points in full generality.
\end{itemize}
\end{prop}

\begin{proof}
We first compute
\begin{align*}
\int_M\Delta\trace h\cdot e^{-f}\dv&=\int_M \trace h(-\Delta f-|\nabla f|^2)e^{-f}\dv, \\
\int_M\delta(\delta h)\cdot e^{-f}\dv&=\int_M\langle h,\nabla^2e^{-f}\rangle\dv=\int_M \left(h(\nabla f,\nabla f)-\langle h,\nabla^2 f\rangle \right)e^{-f}\dv, \\
2\int_M\langle\nabla v,\nabla f\rangle e^{-f}\dv&=2\int_M v(\Delta f+|\nabla f|^2)\dv, 
\end{align*}
where we use that $h$ vanishes to any order at $\partial N$ and $\nabla_{\nu}f=0$. We denote the derivative with respect to $t$ at $t=0$ by a prime. We get
\[ \begin{split}
\frac{d}{dt}& F_{\alpha}(g+th,f+tv)|_{t=0} \\
&=
\int_M \left((\scal')+\alpha(|\nabla f|^2)'\right)\dv
+\int_M(\scal+\alpha|\nabla f|^2)(e^{-f}\dv)'\\
&=
\int_M \left(\Delta\trace h+\delta(\delta h)-\langle \ric,h\rangle-\alpha h(\nabla f,\nabla f)+2\alpha\langle\nabla v,\nabla f\rangle\right)e^{-f}\dv\\
&\qquad +\int_M (\scal+\alpha|\nabla f|^2)(\frac{1}{2}\trace h-v)e^{-f}\dv\\
&=
\int_M \trace h (-\Delta f-|\nabla f|^2)e^{-f}\dv
+\int_M\left(h(\nabla f,\nabla f)-\langle h,\nabla^2 f\rangle\right)e^{-f}\dv\\
&\qquad -\int_M\langle\ric,h\rangle e^{-f}\dv
-\alpha\int_M h(\nabla f,\nabla f)e^{-f}\dv\\
&\qquad +2\alpha\int_M v(\Delta f+|\nabla f|^2)e^{-f}\dv
+\int_M (\scal+\alpha|\nabla f|^2)(\frac{1}{2}\trace h-v)e^{-f}\dv\\
&= 
-\int_M\langle \ric+\nabla^2f,h\rangle e^{-f}\dv
+(1-\alpha)\int_M h(\nabla f,\nabla f)e^{-f}\dv\\
&\qquad+\int_M \trace h(-\Delta f-|\nabla f|^2)e^{-f}\dv
+\alpha\int_M v(2\Delta f+2|\nabla f|^2)\dv\\
&\qquad -\int_M (\scal+\alpha |\nabla f|^2)ve^{-f}\dv
+\frac{1}{2}\int_M \trace h(\scal+\alpha |\nabla f|^2)e^{-f}\dv\\
&=
-\int_M\langle \ric+\nabla^2f-(\alpha-1)\nabla f\otimes\nabla f,h\rangle e^{-f}\dv\\
&\qquad+\int_M (2\alpha\Delta f+\alpha |\nabla f|^2-\scal)ve^{-f}\dv\\
&\qquad +\int_M \trace h(-\Delta f+(\frac{\alpha}{2}-1)|\nabla f|^2+\frac{1}{2}\scal)e^{-f}\dv\\
&=
-\int_M\langle \ric+\nabla^2f-(\alpha-1)\nabla f\otimes\nabla f,h\rangle e^{-f}\dv-\lambda_{\alpha}(g)\int_M ve^{-f}\dv\\
&\qquad +\frac{1}{2\alpha}\int_M \trace h(-2\alpha\Delta f-\alpha |\nabla f|^2+\scal)e^{-f}\dv\\
&\qquad +\frac{1}{2}(1-\frac{1}{\alpha})\int_M \trace^g h\cdot \scal e^{-f}\dv+\frac{1}{2}(\alpha-1)\int_M \trace h |\nabla f|^2e^{-f}\dv\\
&=
-\int_M\langle \ric+\nabla^2f-(\alpha-1)\nabla f\otimes\nabla f,h\rangle e^{-f}\dv\\
&\qquad -\frac{1}{2}\lambda_{\alpha}(g)\int_M \trace h\cdot e^{-f}\dv 
+\lambda_{\alpha}(g)\frac{1}{2\alpha}\int_M \trace h\cdot e^{-f}\dv\\
&\qquad +\frac{1}{2}(1-\frac{1}{\alpha})\int_M \trace h\cdot \scal e^{-f}\dv+\frac{1}{2}(\alpha-1)\int_M \trace h |\nabla f|^2e^{-f}\dv\\
&=
-\int_M\langle \ric+\nabla^2f-(\alpha-1)\nabla f\otimes\nabla f,h\rangle e^{-f}\dv\\
&\qquad +\int_M \langle \frac{1}{2}(1-\frac{1}{\alpha})(\scal-\lambda_{\alpha}(g))g+\frac{1}{2}(\alpha-1)|\nabla f|^2g,h\rangle e^{-f}\dv,
\end{split} \]
which is the first variation formula. Let us denote the weighted $L^2$-gradient of 
$\lambda_{\alpha}$ by $\gradient\lambda_{\alpha}$, that is
\begin{equation}\label{eq:gradient_lambda_general}
\begin{split}
\gradient\lambda_{\alpha}&=
-\left(\ric+\nabla^2f-(\alpha-1)\nabla f\otimes\nabla f\right)\\ &\qquad+\frac{1}{2}(1-\frac{1}{\alpha})(\scal^g-\lambda_{\alpha}(g))g+\frac{1}{2}(\alpha-1)|\nabla f|^2g.
\end{split}
\end{equation}
If $g$ has constant scalar curvature then $\scal=\lambda_{\alpha}(g)$ so $f$ is constant as well and $\int_M e^{-f}\dv=e^{-f}\volume(M,g)=1$. Then $\gradient\lambda_{\alpha}=-\ric$. If $h=u g$, we get
\begin{align*}
D_g{\lambda_\alpha}(u g)
=-\int_M \scal u e^{-f}\dv
=-\frac{\scal}{\volume(M,g)} \int_M u \dv,  
\end{align*}
and the right hand side vanishes if $h=u g$ is volume-preserving. If moreover $\ric=\sigma g$, then
\begin{align*}
D_g{\lambda_\alpha}(h)
=-\int_M\langle \sigma g,h\rangle e^{-f}\dv
=-\sigma\int_M \trace h e^{-f}\dv
=-\frac{\sigma}{\volume(M,g)}\int_M\trace h \dv
\end{align*}
and the right hand side vanishes if $h$ is a volume-preserving deformation of $g$. If $\sigma=0$, it vanishes for all $h$.
\end{proof}

\subsection{The second variation formula}

\begin{prop}\label{prop_second_variation_lambda}
Let $(M,g)$ be an Einstein manifold, $\ric=\sigma g$. Then, the second variation of $\lambda_{\alpha}$ in the direction of volume-preserving deformations $h$ is given by
\[\begin{split}
D_g^2\lambda_{\alpha}(h,h)&=
-\frac{1}{\volume(M,g)}\int_M\langle \frac{1}{2}\Delta_Eh-\delta^*(\delta h)-\frac{1}{2}\delta(\delta h) g,h\rangle \dv\\
&\qquad
-\frac{1}{\volume(M,g)}\int_M v\delta(\delta h) \dv
+ \frac{1}{\volume(M,g)} \int_M \left((\alpha-1)\Delta v+\sigma v\right) \trace h \dv,
\end{split}\]
where $v$ is a solution of the boundary value problem
\[
2\alpha\Delta v = \Delta(\trace h)+\delta(\delta h)-\sigma \trace h, 
\qquad \nabla_{\nu}v=0.
\]
\end{prop}

\begin{proof}
Consider a curve $g_t$ of metrics of constant volume with $g_0=g$ and $\frac{d}{dt} g_t |_{t=0}=h$. Let $v:=\frac{d}{dt} f_{g_t} |_{t=0}$ and $k:=\frac{d^2}{dt^2} g_t |_{t=0}$. We compute
\[ \begin{split}
\frac{d^2}{dt^2}\lambda_{\alpha}(g_t) |_{t=0}
&=\frac{d}{dt}\int_M\langle \grad\lambda_{\alpha}(g_t),h\rangle e^{-f}\dv |_{t=0}\\
&=\int_M\langle (\grad\lambda_{\alpha})',h\rangle e^{-f}\dv
+\int_M\langle \grad\lambda_{\alpha}(g),k\rangle e^{-f}\dv\\
&-2\int_M\langle \grad\lambda_{\alpha},h\circ h\rangle e^{-f}\dv
+\int_M\langle \grad\lambda_{\alpha}(g_t),h\rangle (\frac{1}{2}\trace h-v)e^{-f}\dv.
\end{split} \]
By differentiating \eqref{eq:eulerlagrange_generallambda} and using that $f$ is constant, we see that $v$ satisfies
\begin{equation} \label{eq:formula_v}
2\alpha\Delta v=\Delta(\trace h)+\delta(\delta h)-\langle\ric,h\rangle,
\end{equation}
and since the vector $\nu$ is an outward-pointing normal for all the metrics $g_t$, we conclude that $v$ vanishes on the boundary. Again since $f$ is constant we have 
\[ \begin{split}
(\gradient\lambda_{\alpha})'&=-\ric'-\nabla^2v+\frac{1}{2}(1-\frac{1}{\alpha})\scal' g\\
&=-\left(\frac{1}{2}\Delta_Lh-\delta^*(\delta h)-\frac{1}{2}\nabla^2\trace h\right)-\nabla^2v\\
&\qquad+\frac{1}{2}(1-\frac{1}{\alpha})\left(\Delta(\trace h)+\delta(\delta h)-\langle \ric,h\rangle\right) g\\
&= -\left(\frac{1}{2}\Delta_Lh-\delta^*(\delta h)-\frac{1}{2}\nabla^2\trace h\right)-\nabla^2v+(\alpha-1)\Delta v g,
\end{split} \]
where we used \eqref{eq:formula_v} in the last equation. Since the volume is constant along $g_t$ we get
\[
0=\frac{d^2}{dt^2} \volume(M,g_t) |_{t=0}
=\frac{1}{2}\frac{d}{dt} \int_M\trace \dot{g}_t\dv |_{t=0}
=\frac{1}{2}\int_M\left(\trace k-|h|^2+\frac{1}{2}(\trace h)^2\right)\dv.
\]
From $\gradient\lambda_{\alpha}=-\sigma g$ we thus obtain
\[ \begin{split}
\frac{d^2}{dt^2}\lambda_{\alpha}(g_t) |_{t=0}
&=\int_M\langle (\grad\lambda_{\alpha})',h\rangle e^{-f}\dv+\sigma\int_M\trace k\cdot e^{-f}\dv+2\sigma\int_M |h|^2e^{-f}\dv\\
&\qquad -\frac{\sigma}{2}\int_M(\trace h)^2e^{-f}\dv+\sigma\int_M\trace h\cdot v e^{-f}\dv\\
&=\int_M\langle (\grad\lambda_{\alpha})',h\rangle e^{-f}\dv
+\sigma\int_M |h|^2e^{-f}\dv+\sigma\int_M\trace h\cdot v e^{-f}\dv\\
&=-\int_M\langle \frac{1}{2}\Delta_Eh-\delta^*(\delta h)-\frac{1}{2}\nabla^2\trace h,h\rangle e^{-f}\dv\\
&\qquad-\int_M\langle \nabla^2v,h\rangle e^{-f}\dv+\int_M \trace h  \left((\alpha-1)\Delta v+\sigma v e^{-f} \right) \dv.
\end{split} \]
Finally, recall that the formal adjoint of $\nabla^2 $ is $ \delta\circ\delta$. The result now follows from integration by parts and the fact that $e^{-f}=\volume(M,g)^{-1}$.
\end{proof}

\begin{prop}\label{prop:conformal_second_variation}
Let $(M,g)$ be a manifold of constant scalar curvature. Then the second variation of $\lambda_{\alpha}$ in the direction of volume-preserving conformal deformations $ug$ is given by
\[ \begin{split}
D^2_g\lambda_{\alpha}(ug,ug)
&= 
-\frac{1}{\volume(M,g)}\int_M \left((n-1)\Delta u-\scal u\right)u \dv\\
&\qquad +
\frac{1}{\volume(M,g)}\int_M (\left(n\alpha-(n-1)\right)\Delta v+\scal v)u \dv,
\end{split} \]
where $v$ is a solution of the Neumann boundary value problem
\[
2\alpha\Delta v=(n-1)\Delta u-\scal u, \qquad \nabla_{\nu}v=0.
\]
\end{prop}

\begin{rem}\label{rem:conformal_second_variation}
We may choose the function $v$ as
\[
v=(2\alpha\Delta)^{-1}\left((n-1)\Delta-\scal\right)u,
\]
where $\Delta^{-1}$ denotes the inverse of the Laplacian with Neumann boundary conditions acting on the space of functions with vanishing integral. A careful rearranging of the terms appearing above shows that the second variation can be written as
\[
D^2_g\lambda_{\alpha}(ug,ug)= \frac{1}{\volume(M,g)}\int_M (Lu) u\dv,
\]
where the operator $L$ is defined by
\[
L := \left((n-1)\Delta-\scal\right)\left(((n-2)\alpha-(n-1))\Delta+\scal\right)\Delta^{-1}.
\]
\end{rem}

\begin{proof}Let $h,k$ and $v$ as in the proof of Proposition \ref{prop_second_variation_lambda}. As there, we have
\[ \begin{split}
&\frac{d^2}{dt^2}\lambda_{\alpha}(g_t) |_{t=0}
=\int_M\langle (\grad\lambda_{\alpha})',h\rangle e^{-f}\dv
+\int_M\langle \grad\lambda_{\alpha}(g),k\rangle e^{-f}\dv\\
&\qquad\qquad
-2\int_M\langle \grad\lambda_{\alpha},h\circ h\rangle e^{-f}\dv
+\int_M\langle \grad\lambda_{\alpha}(g_t),h\rangle (\frac{1}{2}\trace h-v)e^{-f}\dv.
\end{split} \]
Again due to the volume constraint,
\[
0=\frac{1}{2}\int_M\left(\trace k-|h|^2+\frac{1}{2}(\trace h)^2\right)\dv.
\]
Since we only consider conformal variatons, we have $h=\frac{1}{n}(\trace h)g$ and $k=\frac{1}{n}(\trace k)g$. Since $g$ is of constant scalar curvature, $\grad\lambda_{\alpha}(g)=-\ric$. Using these facts, we see that
\[ \begin{split}
&\int_M \langle \grad\lambda_{\alpha}(g),k-2 h\circ h+\frac{1}{2}\trace h\cdot h\rangle e^{-f}\dv\\
&\qquad=
-\frac{\scal}{n}\int_M \left(\trace k-2|h|^2+\frac{1}{2}(\trace h)^2\right)e^{-f}\dv=
\frac{\scal}{n}\int_M |h|^2e^{-f}\dv,
\end{split} \]
and
\[
\frac{d^2}{dt^2}\lambda_{\alpha}(g_t) |_{t=0}
=\int_M\langle (\grad\lambda_{\alpha})',h\rangle e^{-f}\dv
+\frac{\scal}{n}\int_M \left(|h|^2+ \trace hv\right) e^{-f}\dv.
\]
Let us now consider the terms here in more detail. As in the previous proof, we compute
\[
(\gradient\lambda_{\alpha})'
= -\left(\frac{1}{2}\Delta_Lh-\delta^*(\delta h)-\frac{1}{2}\nabla^2\trace h\right)-\nabla^2v+(\alpha-1)\Delta v\cdot g.
\]
With $h=ug$ we get
\[
(\gradient\lambda_{\alpha})'
= -\left(\frac{1}{2}(\Delta u)g+(1-\frac{n}{2})\nabla^2u\right)-\nabla^2v+(\alpha-1)\Delta v\cdot g,
\]
which yields
\[
\int_M\langle (\grad\lambda_{\alpha})',ug\rangle e^{-f}\dv
=-(n-1)\int_M (\Delta u) u
+\left(n\alpha-(n-1)\right)\int_M (\Delta v) u e^{-f}\dv.
\]
Furthermore,
\[
\frac{\scal}{n}\int_M \left(|h|^2+ \trace hv\right) e^{-f}\dv=
\scal\int_M \left(u^2+uv\right)e^{-f}\dv.
\]
Recall that $e^{-f}\equiv\volume(M,g)^{-1}$ since the scalar curvature is constant.
Adding up and rearranging the terms yields the desired formula. The formula for $v$ follows from inserting $h=ug$ in \eqref{eq:formula_v}.
\end{proof}

\subsection{Estimates on variations}

\begin{lem}\label{lem:error_term_taylor}
For all $g\in\mathcal{M}^{\infty}_{\hat{g}}$, we have the estimates
\begin{align*}
|D^2_g\lambda_{\alpha}(h,k)|&\leq C\left\|h\right\|_{H^1}\left\|k\right\|_{H^1}, \\
|D^3_g\lambda_{\alpha}(h,h,h)|&\leq C\left\|h\right\|_{C^{2,\alpha}}\left\|h\right\|_{H^1},
\end{align*}
where the constant $C$ can be chosen uniformly for all $g$ in a given small $C^{2,\alpha}$-neighborhood around some fixed metric $\hat{g}$.
\end{lem}
\begin{proof}
The proof of these lemmas is almost identical to the proofs of Propositions 4.3 and 4.5 in \cite{Kroenckearxiv2013} which build on Lemmas 4.2 and 4.4 in the same paper. The only slight difference is how elliptic regularity is applied. Let 
\[
\tilde{g}_t=\tilde{g}+th, \quad 
f_t=f_{\tilde{g}_t}, \quad 
v=\frac{d}{dt}|_{t=0}f_{g_t},\quad 
w=\frac{d^2}{dt^2}|_{t=0}f_{g_t},\quad 
h=\frac{d}{dt}|_{t=0}g_t.
\]
By differentiating \eqref{eq:eulerlagrange_generallambda} once and twice, we get equations of the form
\begin{align*}
\Delta_{f}v&:=\Delta v+\langle \nabla f,\nabla v\rangle=(*),\\
\Delta_{f}w&:=\Delta w+\langle \nabla f,\nabla w\rangle=(**),
\end{align*}
where the right hand sides do not contain derivatives of $v,w$ respectively. 
Since all $f_t$ satisfy the Neumann boundary condition with respect to a unit normal which is the same for all $g_t$, we see that $v$ and $w$ satisfy the same Neumann boundary condition. By differentiating the constraint on the minimizers $f=f_g$, we get
\begin{align*}
\int_M (v-\frac{1}{2}\trace h) e^{-f}\dv&=0,\\
\int_M \left(w+\frac{1}{2} | h|^2-(v-\frac{1}{2}\trace h)^2\right)e^{-f}\dv&=0.
\end{align*}
Now consider the space
\[
V:=\left\{ u:M\to \mathbb{R}\mid \int_M u e^{-f}\dv=0 \right\},
\]
which contains the functions
\[
\widetilde{v}=v-\frac{1}{2}\int_M\trace h\cdot e^{-f}\dv,\qquad
\widetilde{w}=w+\int_M\left(\frac{1}{2}|h|^2-(v-\frac{1}{2}\trace h)^2\right)e^{-f}\dv.
\]
From that fact that $\Delta_{f}$ is self-adjoint with respect to the weighted $L^2(e^{-f}g)$-scalar product together with elliptic regularity, we have isomorphisms
\[
\Delta_{f}:C^{2,\alpha}(M)\cap V\to C^{0,\alpha}(M)\cap V,
\qquad 
\Delta_{f}: H^i(M)\cap V\to H^{i-2}(M)\cap V,
\]
for $i=1,2$. We thus get
\[ \begin{split}
\left\|v\right\|_{C^{2,\alpha}}\leq \left\|\widetilde{v}\right\|_{C^{2,\alpha}}+\left\|v-\widetilde{v}\right\|_{C^{2,\alpha}}&\leq \left\|\Delta_{f}\widetilde{v}\right\|_{C^{0,\alpha}}+ \left\|v-\widetilde{v}\right\|_{C^{2,\alpha}}\\
&\leq C\left\|(*)\right\|_{C^{0,\alpha}}+\left\|v-\widetilde{v}\right\|_{C^{2,\alpha}},
\end{split} \]
and similarly for $w$ and the $H^i$-norms. The rest of the proof follows from computations and standard estimates as in \cite[Section 4]{Kroenckearxiv2013}.
\end{proof}

\section{Global scalar curvature ridigity}
\label{sec:global_rigidity}

This section is devoted to the proof of Theorem \ref{thm:rigiditycptEinstein}. For the case of the round sphere, it is well-known that the assertions of the theorem all hold.
Therefore, we assume throughout this section that $(M,\hat{g})$ is a closed Einstein manifold which is not isometric to the round sphere. After suitable scaling, we also assume that $\volume(M,\hat{g})=1$.

For our proof, we need to use suitable local decomposition of the space of metrics. An important decomposition is provided by Ebin's slice theorem \cite{ebin1970manifold} which provides a slice for the action of the diffeomorphism group. Another decomposition is provided by Koiso \cite[Theorem 2.5]{koiso1979decomposition}, who constructs a slice for the action of the group $C^{\infty}_+(M)$, which acts by pointwise multiplication. In the following we construct a slice which is a refinement of both of these approaches. In our setting, it is convenient to work with the $C^{2,\alpha}$-topology and we will therefore deal with the space of Riemannian metrics of $C^{2,\alpha}$-regularity, which we denote by $C^{2,\alpha}(S^2_+M)$. In particular, we will prove Theorem \ref{thm:rigiditycptEinstein} not only for smooth metrics, but also for metrics of $C^{2,\alpha}$-regularity.

\begin{prop}\label{prop:gauged_constant_scalar_curvature}
There exists a $C^{2,\alpha}$-neighborhood $\mathcal{U}\subset C^{2,\alpha}(S^2_+M)$ of the Einstein metric $\hat{g}$
such that the set 
\[
\mathcal{C}
=\mathcal{U}\cap \left\{g\in C^{2,\alpha}(S^2_+M) \mid \scal^{g} \text{is constant},\volume(M,g)=1, \delta^{\hat{g}}g=0 \right\}
\]
is an analytic Banach submanifold with $T_{\hat{g}}\mathcal{C}=C^{2,\alpha}(TT_{\hat{g}})$.
\end{prop}

\begin{proof}
Let
\begin{align*}
V^{k,\alpha}&:=\left\{f\in C^{k,\alpha}(M)\mid \int_M f\dv^{\hat{g}}=0\right\},\\
W^{k,\alpha}&:=\left\{\delta^{\hat{g}} h \mid h\in C^{k+1,\alpha}(S^2M)\right\}
=\left\{\omega\in C^{k,\alpha}(T^*M)\mid \omega\perp_{L^2(\hat{g})}\kernel((\delta^{\hat{g}})^*)\right\},
\end{align*}
and consider the map
\begin{align*}
\Phi: C^{2,\alpha}(S^2_+M)&\to V^{0,\alpha}\oplus\R_+\oplus W^{1,\alpha},\\
g&\mapsto \left(\scal^g-\int_M \scal^g\dv^{\hat{g}},\volume(M,g),\delta^{\hat{g}}g\right).
\end{align*}
Clearly, $\Phi$ is analytic and $\mathcal{C}=\mathcal{U}\cap \Phi^{-1}((0,1,0))$. 
To prove the proposition, it thus suffices to show that $D_{\hat{g}}\Phi$ is surjective and to determine its kernel. Linearizing at $\hat{g}$ yields
\begin{equation}\label{linearization}
\begin{split}
&D_{\hat{g}}\Phi: C^{2,\alpha}(S^2M)\to V^{0,\alpha}\oplus\R\oplus W^{1,\alpha},\\
h&\mapsto \left( \Delta^{\hat{g}}\trace^{\hat{g}}h+\delta^{\hat{g}}(\delta^{\hat{g}}h)-\sigma\trace^{\hat{g}}h+\int_M\sigma\trace^{\hat{g}}h\dv^{\hat{g}},\frac{1}{2}\int_M\trace^{\hat{g}}h\dv^{\hat{g}},\delta^{\hat{g}}h
\right),
\end{split}
\end{equation}
where $\sigma=\frac{1}{n}\scal^{\hat{g}}$ is the Einstein constant of $\hat{g}$. Define the maps
\begin{align*}
L_1:=(n-1)\Delta^{\hat{g}}-n\sigma : V^{2,\alpha}&\to V^{0,\alpha},\\
L_2:=\delta^{\hat{g}}(\delta^{\hat{g}})^*: W^{3,\alpha}&\to W^{1,\alpha}.
\end{align*}
For $c\in \R, u\in V^{2,\alpha}$, and $\omega\in V^{0,\alpha}$, we compute
\[
D_{\hat{g}}\Phi((u+c)\hat{g}+(\delta^{\hat{g}})^*\omega)
=(L_1(u),\frac{n}{2}c,-\nabla u+L_2(\omega)).
\]
Since $(M,\hat{g})$ is not isometric to the round sphere, we have $\frac{n}{n-1}\sigma\notin \spectrum_{+}(\Delta_{\hat{g}})$ by the Lichnerowicz-Obata eigenvalue estimate \cite{Oba62}. Therefore, $L_1$ is an isomorphism. The operator $L_2$ is also an isomorphism since it is self-adjoint and its domain is the orthogonal complement of its kernel. Note also that for any $f\in V^{2,\alpha}$, $\nabla f$ is orthogonal to $\kernel(L_2)$, as $\delta^{\hat{g}}\eta=-\trace^{\hat{g}} (\delta^{\hat{g}})^*\eta=0$ for any $\eta\in\kernel(L_2)$. Thus, for
\[
(v,d,w)\in V^{0,\alpha}\oplus \R \oplus W^{1,\alpha},
\]
we have
\[
D_{\hat{g}}\Phi((L_1^{-1}(v)+\frac{2}{n}d)\hat{g}+(\delta^{\hat{g}})^*(L_2)^{-1}(w+\nabla (L_1)^{-1}(v))=(v,d,w),
\]
which shows that $D_{\hat{g}}\Phi$ is surjective. Therefore, $\mathcal{C}$ is an analytic Banach submanifold. It remains to determine the kernel of $D_{\hat{g}}\Phi$. The inclusion $C^{2,\alpha}(TT_{\hat{g}}) \subset \kernel(D_{\hat{g}})\Phi$ is clear from \eqref{linearization}. Now suppose that $D_{\hat{g}}\Phi(h)=0$. Then by \eqref{linearization}, it follows that $\delta^{\hat{g}}h=0$. Again by the Obata-Lichnerowicz eigenvalue estimate \cite{Oba62}, $\sigma\notin\spectrum_+(\Delta_{\hat{g}})$. Therefore, $\trace^{\hat{g}}h=0$ and $h\in C^{2,\alpha}(TT_{\hat{g}})$.
\end{proof}

\begin{prop} \label{prop:slicethm}
Provided that the neighborhood $\mathcal{U}$ of $\hat{g}$ in Proposition \ref{prop:gauged_constant_scalar_curvature} is chosen sufficiently small, there exists another $C^{2,\alpha}$-neighborhood $\mathcal{V}\subset C^{2,\alpha}(M)$ of $1$ such that
\[
\Psi: \mathcal{V}\times\mathcal{C}\to \mathcal{U}, \qquad \Psi(f,g) = fg,
\]
is a diffeomorphism onto its image. Furthermore, any metric in $\mathcal{U}$ is isometric to a metric in the analytic submanifold 
\[
\mathcal{S}:=\Psi(\mathcal{V}\times\mathcal{C}).
\]
\end{prop}

\begin{proof}
By \cite[Lemma 4.57]{Besse} we have a direct sum
\[
C^{2,\alpha}(M) \cdot \hat{g} \oplus C^{2,\alpha}(TT_{\hat{g}})=
T_{\hat{g}}(C^{2,\alpha}_+(M)\hat{g})\cdot\hat{g}\oplus T_{\hat{g}}\mathcal{C},
\]
which is the injective image of $D_{(1,\hat{g})}\Psi$. The first assertion follows from the implicit function theorem. For the second assertion, consider the map
\[
\Theta:\mathcal{S}\times \operatorname{Diff}^{2,\alpha}(M)\to C^{2,\alpha}(S^2_+M),
\qquad \Theta(g,\varphi) = \varphi^*g.
\]
The differential $D_{(1,\hat{g})}\Theta$ corresponds to the decomposition
\[ \begin{split}
C^{2,\alpha}(S^2M)
&=C^{2,\alpha}(M)\cdot \hat{g}\oplus 
C^{2,\alpha}(TT_{\hat{g}})\oplus \delta^*_{\hat{g}}(C^{3,\alpha}(T^*M))\\
&=T_{\hat{g}}S\oplus T_{\hat{g}}( \hat{g}\cdot \operatorname{Diff}^{2,\alpha}(M))
\end{split} \]
and thus, it is surjective. Consequently, $\Theta$ is surjective near $\hat{g}$ which proves the second assertion.
\end{proof}

\begin{prop}\label{prop:finite-dim_moduli_space}
Provided that $\mathcal{U}$ is chosen small enough, there exists a real analytic finite-dimensional submanifold $\mathcal{Z}\subset\mathcal{C}$ whose tangent space is equal to $T_{\hat{g}}\mathcal{Z}=\ker(\Delta_E)\cap C^{2,\alpha}(TT_{\hat{g}})$ such that
\[
\mathcal{E}:=
\left\{g\in\mathcal{C}\mid \ric^{g}=\frac{\scal^g}{n}g\right\}\subset \mathcal{Z}
\]
is a real analytic subset.
\end{prop}

\begin{proof}
A similar statement is given in \cite[Lemma 13.6]{Koiso83} for maps between Hilbert spaces. The same proof however also works in our more general setting. We give the proof for the sake of completeness. Consider the real analytic map
\[
\Xi: \mathcal{C}\to C^{0,\alpha}(S^2M), \qquad \Xi(g) = \ric^{g}-\frac{\scal^g}{n}g.
\]
Clearly, $\Xi(\hat{g})=0$ and the differential of $\Xi$ at $\hat{g}$ is
\[
D_{\hat{g}}\Xi :TT_{\hat{g}}\to C^{0,\alpha}(S^2M), 
\qquad D_{\hat{g}}\Xi(h)= \Delta_Eh.
\]
Consider the closed subspace
\[
V:=\operatorname{im}(D_{\hat{g}}\Xi)
=\left\{h\in C^{0,\alpha}(TT_{\hat{g}})\mid h\perp_{L^2(\hat{g})} \ker(\Delta_E)\cap C^{2,\alpha}(TT_{\hat{g}})
\right\}
\]
and the $L^2(\hat{g})$-orthogonal projection $\pi_V: C^{0,\alpha}(S^2M)\to V$. By construction, the map $\pi_V \circ \Xi:\mathcal{C}\to V$ has surjective differential $D_{\hat{g}}(\pi_V \circ \Xi)$. Thus by the implicit function theorem,
\[
\mathcal{Z}:=(\pi_V \circ \Xi)^{-1}(0)
\]
is a real analytic submanifold of $\mathcal{C}$ such that
\[
T_{\hat{g}}\mathcal{Z}
= \ker(\Delta_E)\cap C^{2,\alpha}(TT_{\hat{g}}),
\]
in particular, $\mathcal{Z}$ is finite-dimensional. Finally,
\[
\mathcal{E} = \Xi^{-1}(0)
=(\pi_V\circ\Xi)^{-1}(0)\cap (1-\pi_V\circ\Xi)^{-1}(0)
=\mathcal{Z}\cap (1-\pi_V\circ\Xi)^{-1}(0)
\subset \mathcal{Z}
\]
is an analytic subset.
\end{proof}

\begin{proof}[Proof of Theorem \ref{thm:rigiditycptEinstein}]
Throughout the proof, we assume that the $C^{2,\alpha}$-neighborhood s $\mathcal{U}, \mathcal{C}$, and $\mathcal{V}$ are chosen so small that the above propositions apply. By Proposition \ref{prop:slicethm} and diffeomorphism invariance, the equivalences hold in general if they hold for all metrics in $\mathcal{S}$. Also by Proposition \ref{prop:slicethm}, we may write any metric in $\mathcal{S}$ as $fg$ with $f\in\mathcal{V}$ and $g\in \mathcal{C}$.
By \cite[Theorem C]{BWZ04}, any $g\in \mathcal{C}$ is a Yamabe metric. Thus,
\begin{equation}\label{eq:Yambeconstscalar}
Y(fg)=Y(M,[fg])=Y(M,[g])=\scal^g
\end{equation}
for all $fg\in\mathcal{S}$. Now we are ready to prove the desired implications.\\
\textbf{(i)$\Leftrightarrow$(iii):} This is an immediate consequence of \eqref{eq:Yambeconstscalar}.\\
\textbf{(ii)$\Rightarrow$(iii):} This is trivial.\\
\textbf{(iii)$\Rightarrow$(ii):} By the Obata-Lichnerowicz eigenvalue estimate \cite{Oba62}, we know that $\spectrum_+(\Delta_{\hat{g}})>\frac{\scal_{\hat{g}}}{n-1}$. By continuous dependence of eigenvalues, we may therefore choose a constant $\alpha>0$ such that
\begin{equation}\label{eq:assumption_alpha}
\left(1-\frac{n-2}{n-1}\alpha\right)\spectrum_+(\Delta_g) > \frac{\scal_g}{n-1}
\end{equation}
for all $g\in\mathcal{C}$, provided that $\mathcal{C}$ is sufficiently small around $\hat{g}$. We are going to show that $\hat{g}$ is a maximum of $\lambda_{\alpha}$ on $\mathcal{S}\cap \mathcal{M}_1$, which by Lemma \ref{lem:lambda_and_rigidity} implies the nonexistence of metrics $g\in\mathcal{S}\cap\mathcal{M}_1$ with $\scal^g \geq \scal^{\hat{g}}$ everywhere and $\scal^{g}>\scal^{\hat{g}}$ somewhere. Let now $fg\in\mathcal{S}\cap \mathcal{M}_1$. By setting
\[
u:=\frac{ f-\int_M f\dv^g}{\int_Mf\dv^g},\qquad d:=\int_M f\dv^g,
\]
we can write $fg=d(1+u)g$, where $u$ is a function with $\int_M u\dv_g=0$. 
Next, we join $g$ and $fg$ through
\[
g_t:=\volume(M,(1+tu)g)^{-\frac{n}{2}}(1+tu)g,\qquad t\in[0,1],
\]
which is a curve of metrics of unit volume. Note that $d=\volume(M,(1+u)g)^{-\frac{n}{2}}$ since $fg=d(1+u)g$ is assumed to be of unit volume. Therefore, $g_0=g$ and $g_1=fg$. By Taylor expansion, we have
\[
\lambda_{\alpha}(fg) =\lambda(g) +\frac{d}{dt}\lambda_{\alpha}(g_t) |_{t=0}
+\frac{1}{2}\frac{d^2}{dt^2}\lambda_{\alpha}(g_t) |_{t=0} + R(g,u),
\]
where
\[
R(g,u) :=\frac{1}{2}\int_0^1(1-t)^2\frac{d^3}{dt^3}\lambda_{\alpha}(g_t)dt.
\]
Since $\int_M u\dv_g=0$, we get
\[
\frac{d}{dt}g_t |_{t=0} = ug.
\]
Proposition \ref{prop_first_variation_general_lambda} then implies
\[
\frac{d}{dt}\lambda_{\alpha}(g_t) |_{t=0} = D(\lambda_{\alpha})_g (ug)=0, 
\]
and from Remark \ref{rem:conformal_second_variation} we know that
\[
\frac{d^2}{dt^2}\lambda_{\alpha}(g_t) |_{t=0} =-\int_M (L^g u) u \dv_g,
\]
where
\[
L^g =(n-1)\left((n-1)\Delta^g-\scal^g\right)\left(-\left(1-\frac{n-2}{n-1}\alpha\right)\Delta^g+\frac{\scal^g}{n-1}\right)(\Delta^g)^{-1}.
\]
By \eqref{eq:assumption_alpha}, we may choose $C_1$ so small that
\[
L^g<-2C_1\Delta_g
\]
for all $g\in\mathcal{C}$. This implies
\[
\frac{d^2}{dt^2}\lambda_{\alpha}(g_t) |_{t=0}
=-\int_M (L^g u) u \dv^g < -2C_1\left\| u \right\|_{H^1}^2,
\]
where $C_1>0$ does not depend on $g\in\mathcal{C}$. By Lemma \ref{lem:error_term_taylor}, we have
\[
\frac{d^3}{dt^3}\lambda_{\alpha}(g_t)
\leq C_2\left\|u\right\|_{C^{2,\alpha}}\left\| u\right\|_{H^1}^2,
\]
which immediately yields
\[
R(g,u)\leq C_3\left\|u\right\|_{C^{2,\alpha}}\left\| u\right\|_{H^1}^2.
\]
Combining all these estimates we find that
\[
\lambda_{\alpha}(fg) \leq \lambda_{\alpha}(g)-(C_1-C_3\left\|u\right\|_{C^{2,\alpha}})\left\| u\right\|_{H^1}^2\leq \lambda_{\alpha}(g),
\]
provided that $f$ (and hence $u$) lies in a sufficiently small neighborhood $\mathcal{V}$ of $1$ in $C^{2,\alpha}(M)$. Finally, by assuming (iii), we have
\[
\lambda_{\alpha}(fg)\leq \lambda_{\alpha}(g)=\scal^g\leq \scal^{\hat{g}}=\lambda_{\alpha}(\hat{g}),
\]
which is what we needed to show.\\
\textbf{(ii)$\Rightarrow$(iv):} Suppose that $g \in \mathcal{M}_1$ satisfies $\scal^{g}=\scal^{\hat{g}}$ but is not Einstein. We then consider the tracefree tensor $h:=\ric^{g}-\frac{1}{n}\scal^g g$ which does not vanish identically. Since $\scal^{g}$ is constant, the contracted second Bianchi identity implies that
\[
\delta h=\delta\ric^{\tilde{g}}=-\frac{1}{2}\nabla\scal^{\tilde{g}}=0,
\]
and $h$ is thus a TT-tensor. Let $g_t$ be a solution of the volume-normalized Ricci-de~Turck flow
\begin{align*}
\frac{d}{dt} g_t &=
-2\ric^{g_t}+\frac{2}{n \volume(M,g_t)} \int_M\scal^{g_t}\dv^{g_t} g_t
+\mathcal{L}_{V(g_t,g)}g_t,\\
\qquad V(g_t,g)^k &= (g_t)^{ij}(\Gamma(g_t)_{ij}^k-\Gamma(g)_{ij}^k),
\end{align*}
with reference and initial metric $g_0=g$. Then, $g_t$ is a smooth curve in $\mathcal{M}_1$ which satisfies $g_0=\tilde{g}$ and $\frac{d}{dt} g_t |_{t=0} = -2h$. Moreover, by parabolic theory, the metrics $g_t$ are smooth for $t>0$ and stay $C^{2,\alpha}$-close to $\tilde{g}$. Since $h$ is a TT-tensor, the first variation of the scalar curvature is
\[
\frac{d}{dt} \scal^{g_t} |_{t=0} 
= -2\left(\Delta(\trace h)+\delta(\delta h)-\langle \ric^{g_0},h\rangle\right)
= 2\langle h+\frac{1}{n}\scal^g g,h\rangle
= 2|h|^2.
\]
Since $h\not\equiv0$, we get a curve $g_t\in \mathcal{M}_1$ such that for small $t>0$, we have $\scal^{g_t} \geq \scal^{\tilde{g}} = \scal^{g}$ everywhere and $\scal^{g_t}>\scal^{g}$ somewhere, contradicting (ii).\\
\textbf{(iv)$\Rightarrow$(iii):} Let us consider the disjoint sets
\begin{align*}
\mathcal{C}_{>}&:=\left\{{g}\in\mathcal{C}\mid \scal^{{g}}>\scal^{\hat{g}}\right\},\\
\mathcal{C}_{<}&:=\left\{{g}\in\mathcal{C}\mid \scal^{{g}}<\scal^{\hat{g}}\right\},\\
\mathcal{C}_{=}&:=\left\{{g}\in\mathcal{C}\mid \scal^{{g}}=\scal^{\hat{g}}\right\}.
\end{align*}
Recall that $T_{\hat{g}}\mathcal{C}=C^{2,\alpha}(TT_{\hat{g}})$. Since the second variation of the Einstein-Hilbert functional is unbounded below on $TT$-tensors (see for example \cite[Theorem 4.60]{Besse}), we always have that $\mathcal{C}_{<} \neq\emptyset$.
By Proposition \ref{prop:finite-dim_moduli_space}, there is a finite-dimensional submanifold $\mathcal{Z}\subset\mathcal{C}$ which contains $\mathcal{E}$ and thus $\mathcal{C}_{=}$ by assumption. Assume that (iii) does not hold. Then $\mathcal{C}_{>}$ and $\mathcal{C}_{<}$ are both nonempty open subsets of $\mathcal{C}$. We thus find metrics $g_+\in \mathcal{C}_{>}\setminus \mathcal{Z}$ and $g_{-}\in \mathcal{C}_{<}\setminus \mathcal{Z}$. As $\mathcal{Z}$ has infinite codimension in $\mathcal{C}$, $\mathcal{C}\setminus \mathcal{Z}$ is connected. Thus, we can join $g_{+}$ and $g_{-}$ by a continuous curve $g_t$, $t\in [0,1]$ in $\mathcal{C}\setminus \mathcal{Z}$. On the other hand, as the scalar curvature of $g_t$ is continuous in $t$, we must have $g_{t_0}\in \mathcal{C}_{=}\subset \mathcal{Z}$ for some $t_0\in [0,1]$ which is a contradiction.
\end{proof}

\section{Local scalar curvature rigidity}
\label{sec_local_rigidity}

For a compact Riemannian manifold $(M,\hat{g})$ with boundary, we use the notation
\[
 \mathcal{M}_{\hat{g}}^{k+2}:=\left\{g\mid g\text{ Riemannian metric such that }g-\hat{g}\in \psi^2\phi^2\mathring{H}^{k+2}_{\phi,\psi}(S^2M)\right\}.
\]

Our first goal in this section is to prove the following theorem.

\begin{thm}\label{thm:local_maximality_lambda_boundary}
Let $(M,\hat{g})$ be a compact $n$-dimensional Einstein manifold with boundary and assume that the first nonzero Neumann eigenvalue of the Laplacian satisfies
\begin{equation} \label{eq:assumption_neumman_eigenvalue}
\mu_1^{NM}(M,\Delta^{\hat{g}}) > \frac{\scal^{\hat{g}}}{n-1}.
\end{equation}
Choose $\alpha>0$ so small that
\[
\left[1-\frac{n-2}{n-1}\alpha\right]\mu_1^{NM}(M,\Delta^{\hat{g}})
>\frac{\scal^{\hat{g}}}{n-1}.
\]
Then, if the first Dirichlet eigenvalue of the Einstein operator on TT-tensors satisfies
\[
\mu_1^D(M,\Delta_E|_{TT})>0,
\]
there exists a $\psi^2\phi^2H^{k+2}_{\phi,\psi}$-neighbourhood $\mathcal{V}$ of $\hat{g}$ in $\mathcal{M}_{\hat{g}}^{k+2}$ such that $\lambda_{\alpha}(g)\leq \lambda_{\alpha}(\hat{g})$ for every $g\in\mathcal{V}$ with $\mathrm{vol}(M,g)=\mathrm{vol}(M,\hat{g})$. Further, equality holds if and only if $g$ is Einstein.
\end{thm}

We showed this assertion for closed manifolds in the proof of Theorem \ref{thm:rigiditycptEinstein}, in the implication (iii)$\Rightarrow$(ii). There, we used a local decomposition of the space of metrics which relies on closedness. We are not able to use these arguments here. We compensate this by assuming strict positivity of $\Delta_E|_{TT}$ which is a stronger assumption than (iii) in Theorem \ref{thm:rigiditycptEinstein}.

To prove Theorem \ref{thm:local_maximality_lambda_boundary} we need some preparation. 
Let us recall that on a compact manifold with boundary, by compactly supported we mean compactly supported in the interior. Let
\[
C^{\infty}_{c,0}(M) := 
\left\{f\in C^{\infty}_{c}(M)\mid \int_M f\dv^{\hat{g}}=0 \right\}.
\]
We have a direct sum
\begin{equation} \label{eq_cs_decomposition_metrics}
C^{\infty}_{c,0}(M)\hat{g} \oplus \delta^*(C^{\infty}_c(T^*M))\oplus C^{\infty}_{c}(TT).
\end{equation}
Note that all $h$ in this direct sum satisfy the condition $\int_M \trace h\dv_g=0$. That is, as deformations of the metric they preserve the total volume to first order.
\begin{prop} \label{prop:estimate_second_variation}
Let $(M,\hat{g})$ and $\alpha>0$ satisfy the assumptions of 
Theorem~\ref{thm:local_maximality_lambda_boundary}. 
Then, $D^2_{\hat{g}}\lambda_{\alpha}$ is diagonal with respect to the decomposition \eqref{eq_cs_decomposition_metrics}. 
For $h\in \delta^*(C^{\infty}_c(T^*M))$, we have $D^2_{\hat{g}}\lambda_{\alpha}(h,h)=0$.
Furthermore, there is a constant $C>0$ so that
\begin{equation}\label{eq:H^1_estimate}
D^2_{\hat{g}}\lambda_{\alpha}(h,h)<-C\left\| h\right\|_{H^1}^2
\end{equation}
holds for all $h\in C^{\infty}_{c,0}(M)\hat{g}\oplus C^{\infty}_{c}(TT)$.
\end{prop}

\begin{proof}
We first consider the sum
\begin{equation} \label{eq:partial_splitting}
C^{\infty}_{c,0}(M)\hat{g}\oplus C^{\infty}_{c}(TT).
\end{equation}
It is a well-known fact that $\Delta_E$ preserves this sum, see for example \cite[Theorem 4.60]{Besse}. Therefore for $h\in C^{\infty}_c(TT)$ and $u\in C^{\infty}_{c,0}(M)\hat{g}$,  Proposition \ref{prop_second_variation_lambda} directly yields
\begin{equation} \label{eq:diagonality}
D^2_{\hat{g}}\lambda_{\alpha}(u\hat{g}+h,u\hat{g}+h)
=D^2_{\hat{g}}\lambda_{\alpha}(u\hat{g},u\hat{g})+D^2_{\hat{g}}\lambda_{\alpha}(h,h).
\end{equation}
Consequently, $D^2_{\hat{g}}\lambda_{\alpha}$ is diagonal with respect to \eqref{eq:partial_splitting}. By diffeomorphism invariance, $D^2_{\hat{g}}\lambda_{\alpha}(\delta^*\omega,\cdot) = 0$ for all $\omega\in C^{\infty}_{c}(T^*M)$, since $\delta^*\omega=\mathcal{L}_{\omega^{\sharp}}\hat{g}$. Thus, $D^2_{\hat{g}}$ is diagonal with respect to the splitting \eqref{eq_cs_decomposition_metrics} and vanishes identically on $\delta^*(C^{\infty}_c(T^*M))$. 

To finish the proof, it remains to prove the $H^1$-estimate \eqref{eq:H^1_estimate}. By \eqref{eq:diagonality} and the triangle inequality, it suffices to prove the $H^1$-estimates on the parts of \eqref{eq:partial_splitting} separately. Recall that for $h\in  C^{\infty}_c(TT)$, we have 
\[
D^2_{\hat{g}}\lambda_{\alpha}(h,h)
=-\frac{1}{2\volume(\hat{g})}\int_M\langle\Delta_Eh,h\rangle \dv
\]
and for $u \hat{g}\in C^{\infty}_{c,0}(M)\hat{g}$, we have
\[
D^2_{\hat{g}}\lambda_{\alpha}(u \hat{g},u \hat{g})
=\frac{1}{\volume(\hat{g})}\int_M u(Lu) \dv,
\]
where the operator $L$ is 
\[
L = [(n-1)\Delta-\scal_g]\left[((n-2)\alpha-(n-1))\Delta+\scal\right]\Delta^{-1},
\]
see Remark \ref{rem:conformal_second_variation}.
In both cases, continuous dependence of eigenvalues tells us that $\Delta_E-\epsilon\cdot (\Delta+1)$ resp. $-L+\epsilon\cdot (\Delta+1)$ will still be positive on the respective spaces for a small $\epsilon>0$. 
Consequently, we get
\[ \begin{split}
D^2_{\hat{g}}\lambda_{\alpha}(h,h)
&=-\frac{1}{2\volume(\hat{g})}\int_M\langle\Delta_Eh,h\rangle \dv \\
&\leq -\frac{\epsilon}{2\volume(\hat{g})}\int_M \left(|\nabla h|^2+|h|^2\right)\dv
=-\frac{\epsilon}{2\volume(\hat{g})}\left\|h\right\|_{H^1}^2
\end{split} \]
and
\[ \begin{split}
D^2_{\hat{g}}\lambda_{\alpha}(u \hat{g},u \hat{g})
&=\frac{1}{\volume(\hat{g})}\int_M u(Lu) \dv \\
&\leq-\frac{\epsilon}{\volume(\hat{g})}\int_M \left(|\nabla u|^2+|u|^2\right)\dv
=-\frac{\epsilon}{\volume(\hat{g})}\left\|u\right\|_{H^1}^2
\end{split} \]
in the respective cases. This finishes the proof of the proposition.
\end{proof}

Let us now consider the set of Riemannian metrics
\[
\mathcal{R} :=
\left\{g \mid g - \hat{g} \in \psi^2\phi^2\mathring{H}^{k+2}_{\phi,\psi}(S^2 M), 
\volume(M,g)=\volume(M,\hat{g})\right\},
\]
which is a manifold (as the volume functional is clearly a submersion) with tangent space
\[
T_{\hat{g}}\mathcal{R}
=\left\{ h\in \psi^2\phi^2 \mathring{H}^{k+2}_{\phi,\psi}(S^2M)\mid
\int_M\trace h\dv=0 \right\}.
\]
We define an exponential map 
\[
\exp_{\hat{g}}:
T_{\hat{g}}\mathcal{R}\supset \mathcal{B}\to \mathcal{R},\qquad h\mapsto f(h)(\hat{g}+h),
\]
where the conformal factor $f(h)\in \psi^2\phi^2H^{k+2}_{\phi,\psi}(M)$ is determined by the condition
\begin{equation} \label{eq:volume_element_constraint}
\dv_{\exp_{\hat{g}}(h)}=(1+\frac{1}{2}\trace h)\dv_{\hat{g}}.
\end{equation}
Note that $\exp_{\hat{g}}$ maps into $\mathcal{R}$. By differentiating \eqref{eq:volume_element_constraint}, we see that
\[ \begin{split}
0&=\frac{d}{dt}[\dv_{\exp_{\hat{g}}(th)}-(1+\frac{t}{2}\trace h)\dv_{\hat{g}}] |_{t=0} \\
& =\frac{d}{dt}[(f(th))^{n/2}\dv_{\hat{g}+th}] |_{t=0} -\frac{1}{2}\trace h\dv_{\hat{g}}\\
&=\frac{n}{2}\frac{d}{dt} f(th) \dv_{\hat{g}} |_{t=0},
\end{split} \]
where we used that $f(0)=1$.
Thus, $\frac{d}{dt}\exp_{\hat{g}}(th)|_{t=0}=h$, so that
\[
D_{0}\exp_{\hat{g}}=\mathrm{id}_{T_{\hat{g}}\mathcal{R}}.
\]
In particular, $\exp_{\hat{g}}$ is a local diffeomorphism in a neighborhood of the origin. 

Now, let
\begin{equation} \label{eq_def_wt_mc_S}
\widetilde{\mathcal{S}}:= \psi^2\phi^2\mathring{H}^{k+2}_{\phi,\psi}(M)\hat{g}\oplus 
\psi^2\phi^2 \mathring{H}^{k+2}_{\phi,\psi}(TT).
\end{equation}
and 
\[
\widetilde{\mathcal{S}}_{\epsilon}=\left\{h\in\widetilde{\mathcal{S}}\mid \left\| h\right\|_{\psi^2\phi^2H^{k+2}_{\phi,\psi}}<\epsilon\right\}.
\]
For a sufficiently small chosen $\epsilon>0$, the set
\[
\mathcal{S}:=\exp_{\hat{g}}(\mathcal{S}_{\epsilon})
\]
is a smooth manifold with tangent space
\[
T_{\hat{g}}\mathcal{S}=\widetilde{\mathcal{S}}.
\]

\begin{prop}\label{prop:lambdamaximality_gaugedmetrics}
Let $(M,\hat{g})$ and $\alpha>0$ satisfy the assumptions of Theorem \ref{thm:local_maximality_lambda_boundary}.	
Then, provided that $\mathcal{S}$ was chosen small enough, we have $\lambda_{\alpha}({g})\leq \lambda_{\alpha}(\hat{g})$ for all $g\in \mathcal{S}$ and equality holds if and only if ${g}=\hat{g}$.
\end{prop}

\begin{proof}
From Proposition \ref{prop:estimate_second_variation}, we get
\begin{equation} \label{eq_another_H^1_estimate}
D^2_{\hat{g}}\lambda_{\alpha}(h,h)\leq -C_1\left\|h\right\|_{H^1}^2
\end{equation}
for all $h\in C^{\infty}_{c,0}(M)\hat{g}\oplus C^{\infty}_c(TT)$. By Lemma \ref{lem:TTdensity}, the $\psi^2\phi^2H^{k+2}_{\phi,\psi}$-closure of this direct sum is exactly \eqref{eq_def_wt_mc_S}. Thus by density, \eqref{eq_another_H^1_estimate} also holds for $h\in \tilde{\mathcal{S}}$. Write $g\in \mathcal{S}$ as $g=\exp_{\hat{g}}(h)$ for $h\in\widetilde{S}$. Taylor expansion then implies
\[
\lambda_{\alpha}(g)
=\lambda_{\alpha}(\hat{g})+\frac{1}{2}D^2_{\hat{g}}\lambda_{\alpha}(h,h)+\frac{1}{2}\int_0^1(1-t)^2\frac{d^3}{dt^3}\lambda_{\alpha}(g_t)dt,
\]
where $g_t:=\exp_{\hat{g}}(th)$. Now we are estimating the error term in the expansion. By Lemma \ref{lem:error_term_taylor},
\[ \begin{split}
\left| \frac{d^3}{dt^3}\lambda_{\alpha}(g_t) \right|
&= 
\left| D_{g_t}\lambda_{\alpha}(g_t''')
+3D^2_{g_t}\lambda_{\alpha}(g_t',g_t'')
+D^3_{g_t}\lambda_{\alpha}(g_t',g_t',g_t') \right| \\
&\leq 
C \left( \left\| g_t'''\right\|_{L^2} 
+ \left\|g_t'\right\|_{H^1}\left\|g_t''\right\|_{H^1} 
+ \left\|g_t'\right\|^2_{H^1}\left\|g_t'\right\|_{C^{2,\alpha}} \right).
\end{split} \]
From the construction of the exponential map we get pointwise bounds
\[
|\nabla^kg^{(m)}_t|(p)\leq C(m,k)\left(\sup_{0\leq l\leq k}|\nabla^lh|(p)\right)^m
\]	
for all $p\in M$ and $m,k\in \N_0$. Therefore, we can conclude from standard estimates that
\[
\left| \frac{d^3}{dt^3}\lambda_{\alpha}(g_t) \right|
\leq C\left\|h\right\|_{C^{2,\alpha}}\left\|h\right\|_{H^1}^2
\leq C\left\|h\right\|_{\psi^2\phi^2H^{k+2}_{\phi,\psi}}\left\|h\right\|_{H^1}^2,
\]
so that
\begin{equation} \label{eq:H^1_estimate_local_maximality}
\lambda_{\alpha}(g) \leq 
\lambda_{\alpha}(\hat{g}) 
-(C_1-C_2\left\|h\right\|_{C^{2,\alpha}})\left\|h\right\|_{H^1}^2
\leq \lambda_{\alpha}(\hat{g})-\frac{1}{2}C_1\left\|h\right\|_{H^1}^2,
\end{equation}
provided that $\mathcal{S}$ is chosen small enough. All the assertions are now immediate.
\end{proof}

\begin{proof}[Proof of Theorem \ref{thm:local_maximality_lambda_boundary}]
Let $C^{\infty}_{c}(TM)$ be the space of smooth vector fields of compact support and let $\Diff_{c}(M)$ be the group of diffeomorphisms generated by $C^{\infty}_{c}(TM)$. 
Let $\mathcal{S}$ be the set in Proposition \ref{prop:lambdamaximality_gaugedmetrics}.
By this proposition, $\hat{g}$ is a local maximum of $\lambda_{\alpha}$ on $\mathcal{S}$. We have a map
\[
\Phi: {\mathcal{S}}\times \Diff_{c}(M)\to \mathcal{M}_{\hat{g}}^{k+2},
\]
and since $\lambda_{\alpha}$ is diffeomorphism invariant, $\hat{g}$ is a local maximum of $\lambda_{\alpha}$ on the set
\[
\mathcal{W}:=\Phi(\mathcal{S}\times \Diff_{c}(M)).
\]
Let $\overline{\mathcal{W}}$ be the closure of $\mathcal{W}$ with respect to the $\psi^2\phi^2H^{k+2}_{\phi,\psi}$-norm. By continuity, $\hat{g}$ is a local maximum of $\lambda_{\alpha}$ also on $\overline{\mathcal{W}}$. The tangent space $T_{\hat{g}}{\overline{\mathcal{W}}}$ is then the closure of
\[ \begin{split}
T_{\hat{g}}{\mathcal{W}}
&=T_{\hat{g}}\mathcal{S}\oplus \delta^*(C^{\infty}_{c}(T^*M))\\
&=\psi^2\phi^2 \mathring{H}^{k+2}_{\phi,\psi}(M)\hat{g}\oplus 
\psi^2\phi^2\mathring{H}^{k+2}_{\phi,\psi}(TT)\oplus \delta^*(C^{\infty}_{c}(T^*M))
\end{split} \]
with respect to the  $\psi^2\phi^2H^{k+2}_{\phi,\psi}$-norm. We clearly have
\[ \begin{split}
T_{\hat{g}}\overline{\mathcal{W}}
&=\psi^2\phi^2 \mathring{H}^{k+2}_{\phi,\psi}(M)\hat{g}
\oplus \psi^2\phi^2 \mathring{H}^{k+2}_{\phi,\psi}(TT)
\oplus \overline{\delta^*(C^{\infty}_{c}(T^*M))}^{\psi^2\phi^2H^{k+2}_{\phi,\psi}}\\
&=\psi^2\phi^2\mathring{H}^{k+2}_{\phi,\psi}(S^2M).
\end{split} \]
In particular, $\overline{\mathcal{W}}$ contains an open neighbourhood $\mathcal{V}$ of $g$ with respect to the  $\psi^2\phi^2H^{k+2}_{\phi,\psi}$-norm. We already know that $\hat{g}$ is a maximizer of $\lambda_{\alpha}$ on $\mathcal{U}$. To finish the proof, it suffices to show that the maximum can be only attained by Einstein metrics. To prove this suppose that $g$ is a metric in $\mathcal{V}$ such that $\lambda_{\alpha}(g)=\lambda_{\alpha}(\hat{g})$. 
Let $g_i$ be a sequence in $\mathcal{W}$ such that $g_i\to {g}$ in the $\psi^2\phi^2H^{k+2}_{\phi,\psi}$-norm.
By continuity, $\lambda_{\alpha}({g}_i)\to \lambda_{\alpha}(\hat{g})$. By construction of $\mathcal{W}$, there are metrics $\hat{g}_i\in \mathcal{S}$ isometric to $g_i$. By diffeomorphism invariance, $\lambda_{\alpha}(\hat{g}_i)\to \lambda_{\alpha}(\hat{g})$ as well and \eqref{eq:H^1_estimate_local_maximality} implies that 
\[
\hat{g_i}\to \hat{g}\qquad  \text{ in }H^1.
\] 
We thus have
\[
\ric^{\hat{g}_i}-\sigma g_i \to \ric^{\hat{g}}-\sigma \hat{g} =0,\text{ in }H^{-1},
\qquad 
\ric^{{g_i}}-\sigma g_i\to \ric^{{g}}-\sigma g,\text{ in }\psi^2\phi^2H^{k}_{\phi,\psi},
\]
where $\sigma$ is the Einstein constant of $g$. Since the $\psi^2\phi^2H^{k}_{\phi,\psi}$-norm and the $H^1$-norm are stronger than the $H^{-1}$-norm, we also have
\[
g_i\to g,
\qquad \hat{g}_i \to\hat{g},
\qquad \ric^{{g_i}} \to \ric^{{g}},
\qquad \ric^{\hat{g}_i} \to \ric^{\hat{g}} =\sigma \hat{g}
\]
in $H^{-1}$. By diffeomorphism invariance of the norm,
\begin{align*}
\left\|\ric^{{g}}-\sigma {g} \right\|_{H^{-1}({g})}\gets 
&\left\|\ric^{{g_i}}-\sigma  g_i\right\|_{H^{-1}({g}_i)}\\
&=
\left\|\ric^{{\hat{g}_i}}-\sigma \hat{g}_i\right\|_{H^{-1}(\hat{g}_i)} \to 
\left\|\ric^{\hat{g}}-\sigma \hat{g}\right\|_{H^{-1}(\hat{g})}=0.
\end{align*}
Therefore, $g$ is Einstein which was to be shown.
\end{proof}
\begin{rem}\label{rem:maximality_lambda_Ricci_flat}
If $(M,g)$ is Ricci-flat, the assertion of Theorem \ref{thm:local_maximality_lambda_boundary} holds without the volume constraint due to Proposition \ref{prop_first_variation_general_lambda} (iii). The proof is the same up to a slight simplification which comes from dropping the volume contstraint.
\end{rem}
\begin{proof}[Proof of Theorem \ref{thm:rigidityopenEinstein}]
This proof is now a straightforward assembly of the partial results we have obtained so far.\\
\textbf{(i)$\Rightarrow$(ii)\&(iii):} Let $K\subset M$ be a compact subset and let $g$ be a smooth metric on $M$ which is $C^{2,\alpha}$-close to $\hat{g}$ and satisfies
\[
g-\hat{g}|_{M\setminus K}\equiv0.
\]
Further, let $N$ be a compact manifold with boundary such that 
\[
K\subset N\subset M.
\]
Then, $g|_N$ is $\psi^2\phi^2H^{k+2}_{\phi,\psi}$-close to $\hat{g}|_N$.
Choose $N$ so close to $M$ that 
\[
\mu_1^{NM}(N,\Delta^{\hat{g}})>\frac{\scal^{\hat{g}}}{n-1}
\]
and choose $\alpha>0$ so small that
\[
\left[1-\frac{n-2}{n-1}\alpha\right]\mu_1^{NM}(N,\Delta^{\hat{g}})
>\frac{\scal^{\hat{g}}}{n-1}.
\]
From (i), $(M,g)$ is linearly stable, and we have $\lambda^D_1(N,\Delta_E|_{TT})>0$. Thus by Theorem \ref{thm:local_maximality_lambda_boundary}, we get $\lambda_{\alpha}(g|_N)\leq \lambda_{\alpha}(\hat{g}|_N)$, provided that $g$ was chosen sufficiently close to $\hat{g}$. Thus by Lemma \ref{lem:lambda_and_rigidity}, it can not happen that $\scal^{g}\geq \scal^{\hat{g}}$ everywhere and $\scal^{g}>\scal^{\hat{g}}$ somewhere. This implies (ii). Now suppose that $\scal^{g}\equiv \scal^{\hat{g}} \equiv \text{constant}$. Then we have $\lambda_{\alpha}(g|_N)= \lambda_{\alpha}(\hat{g}|_N)$ and the rigidity part of Lemma \ref{lem:lambda_and_rigidity} implies that $g$ is Einstein as well. Now, $g$ and $\hat{g}$ are both Einstein metrics and hence both analytic
Thus, since $g$ and $\hat{g}$ agree on an open set, they have to be isometric, see \cite{deturck1981some}. \\
\textbf{(ii)$\Rightarrow$(i)} Suppose that (i) does not hold. Then, we find a nontrivial TT-tensor $h$ with compact support $K:=\supp(h)\subset M$ such that
\[
\int_M\langle \Delta_E h,h\rangle \dv<0.
\]
Let $N$ be a smooth manifold with smooth boundary such that $K\subset N\subset M$. Further, let $f\in C^{\infty}(M)$ be a nonnegative function with $\supp(f)\subset N$ and 
\[
\int_M\langle \Delta_E h,h\rangle \dv=-2\int_Mf\dv<0.
\]
Then by Theorem \ref{thm-prescribe-scal}, there exists a family of metrics $g_t$ such that 
\[
\scal^{g_t}=\scal^{\hat{g}}+\frac{t^2}{2}f.
\]
Since $f$ nonnegative but not identically vanishing, this contradicts (ii).\\
\textbf{(iii)$\Rightarrow$(i)}
This proof is similar to the previous one. If (i) does not hold, we find a TT-tensor $h_-$ with compact support $K:=\supp(h)\subset M$ such that
\[
\int_M\langle \Delta_E h_-,h_-\rangle \dv<0.
\]
On the other hand, because $\Delta_E$ is unbounded from above, we find another TT-tensor $h_+$ supported in $K$ such that 
\[
\int_M\langle \Delta_E h_+,h_+\rangle \dv>0.
\]
 By continuity, there exists $s\in (0,1)$ such that $h:=(1-s)h_-+sh_+$ satisfies
\[
\int_M\langle \Delta_E h,h\rangle \dv=0.
\]
Because $h_-$ and $h_+$ are linearly independent, $h$ is a nontrivial TT-tensor supported in $K$.
Thus by Theorem \ref{thm-prescribe-scal}, there exists a family of metrics $g_t$ with $\frac{d}{dt}g_t|_{t=0}=h$ such that 
\[
\scal^{g_t}=\scal^{\hat{g}}.
\]
It remains to show that the metrics $g_t$ are not Einstein for small values of $t$. Since $h$ is a TT-tensor we have
\[
\frac{d}{dt}\left(\ric^{g_t}-\frac{\scal^{g_t}}{n}g_t)\right) |_{t=0}
=\frac{1}{2}\Delta_Eh.
\]
However, as $h$ is compactly supported, the unique continuation property for elliptic equations tells us that $h\notin \ker(\Delta_E)$. Therefore, $g_t$ can not be Einstein for small $t$, which contradicts (iii).
\end{proof}
\begin{proof}[Proof of Theorem \ref{thm:rigidityopenRicciflat}]
Up to taking Remark \ref{rem:prescribing_scal_Ricci_flat} and Remark \ref{rem:maximality_lambda_Ricci_flat} into account, the proof is the same.
\end{proof}

\section{Mass-decreasing perturbations}

Two well-known complete Ricci-flat manifolds which are linearly unstable are the Riemannian Schwarzschild manifold and Taub-Bolt manifold. 

Let $\sigma$ be the standard round metric on $S^{2}$.
For $m > 0$, the Riemannian Schwarzschild metric is defined on 
$S^1(8\pi m) \times (2m,\infty) \times S^2$ by
\begin{equation} \label{R-S}
g^{RS} = 
\left( 1 - \frac{2m}{r} \right) dt^2 + \left( 1 - \frac{2m}{r} \right)^{-1} dr^2 + r^2 \sigma,
\end{equation}
where $t \in S^1(8\pi m)$ which is the circle of length $8\pi m$, and $r \in (2m,\infty)$.
It extends to $r=2m$ to give a complete Ricci-flat metric on $\R^2 \times S^2$. The metric $g^{RS}$ has an asymptotically flat end, meaning that the metric approaches the flat product metric on $\R^3 \times S^1$ on this end. This metric is known to be linearly unstable, see \cite{Gross-et-al}, \cite{Allen}, \cite{Takahashi}.

Let $\sigma_1, \sigma_2, \sigma_3$ be the standard left invariant 1-forms on $S^3$. 
For $m > 0$, the Taub-Bolt metric is defined on $(2m,\infty) \times S^3$
\begin{equation} \label{taub-bolt}
g^{TB} = 
\frac{r^2 - m^2}{r^2 - \frac{5}{2}mr + m^2} dr^2
+ (r^2 - m^2) (\sigma_1^2 + \sigma_2^2)
+ 4m^2 \frac{r^2 - \frac{5}{2}mr + m^2}{r^2 - m^2} \sigma_3^2,
\end{equation}
where $r \in (2m,\infty)$.
It extends to $r=2m$ to give a complete Ricci-flat metric on $\C P^2 \setminus \{\text{point}\}$ with an asymptotically locally flat end, which means that the metric approaches a flat metric on the total space of a non-trivial $S^1$-fibration over $\R^3 \setminus B^3$ on the end. The metric $g^{TB}$ is also known to be linearly unstable, see \cite{Young}, \cite{Warnick}, \cite{Holzegel-et-al}.

\begin{prop}
The Riemannian Schwarzschild and the Taub-Bolt manifolds allow compactly supported perturbations which strictly increase scalar curvature.
\end{prop}

\begin{proof}
These metrics have a negative bottom of the $L^2$-spectrum of $\Delta_L$ on the whole manifold. This implies that lowest Dirichlet eigenvalue of $\Delta_L$ on any sufficiently large open subset $\Omega$ with smooth boundary is negative. The corresponding eigensection is a TT-tensor, and from Corollary \ref{cor-neg-cpt-supp} we conclude that $g^{RS}$ and $g^{TB}$ are unstable as in Definition \ref{def-stable-open} on such $\Omega$.
The existence of compactly supported perturbations which increase scalar curvature now follows from Theorem \ref{thm-prescribe-scal}.
\end{proof}

For asymptotically flat and asymptotically locally flat manifolds there is a mass invariant defined at infinity, similar to the ADM mass for asympotically Euclidean manifolds. 
We will now see that compactly supported deformations from the above proposition can be transformed to mass-decreasing scalar flat perturbations which preserve the length of the circle at infinity. 

When discussing the mass invariant we restrict attention to manifolds of dimension four.
Let $h_0$ be the standard flat product metric on $\R^3 \times S^1$, where the circle factor has length $L$. Following \cite[Section~3.3]{MinerbeMassALF} we define $(M,g)$ to be asymptotically flat if outside compact subsets, $M$ is diffeomorphich to $\R^3 \times S^1$, and under this diffeomorphism it holds that 
\[
g = h_0 + O(r^{-\tau}), \qquad
\partial_i g = O(r^{-\tau-1}), \qquad
\partial_i\partial_j g = O(r^{-\tau-2}),
\]
for some $\tau > 1/2$. For such manifolds, the mass is defined by 
\[
\mu^D_{g}
:= 
\frac{1}{4\pi L}
\lim_{R \to \infty} \int_{S_R}
\left( \delta^{h_0} g - d \trace^{h_0} g \right) 
\lrcorner \dv^{h_0}.
\]

Using a conformal trick from \cite[Lemma~3.3]{SchoenYauPMT} we can now prove the following.

\begin{prop}
The Riemannian Schwarzschild manifold has scalar-flat perturbations which strictly decrease the mass, while keeping the length $L$ at infinity constant.
\end{prop}

\begin{proof}
Let $g_t$ be a family of metrics on $M=\R^2 \times S^2$ with $g_0 = g^{RS}$, $g_t = g^{RS}$ outside a compact set, and $\scal^{g_t} \geq 0$ with $\scal^{g_t} > 0$ somewhere for $t>0$.

Let $\varphi_t = 1 + u_t$ be the solution to 
\[
0 = - 6\Delta^{g_t} \varphi_t + \scal^{g_t} \varphi_t = 
- 6\Delta^{g_t} u_t + \scal^{g_t} u + \scal^{g_t}
\]
with $u_t \to 0$ at infinity. The existence and uniqueness of such a $u_t$ can be deduced as in \cite[Section~2]{MinerbeMassALF}.
By the maximum principle we have $\varphi_t > 0$. Set $\tilde{g}_t := \varphi_t^2 g_t$. Then $(M,\tilde{g}_t)$ is an AF manifold with $\scal^{\tilde{g}_t} = 0$. The length $L$ of the circle factor at infinity for $\tilde{g}_t$ is the same as for $g^{RS}$.

We compute the mass of $(M,\tilde{g}_t)$,
\[ \begin{split}
\mu^D_{\tilde{g}_t}
&= 
\frac{1}{4\pi L}
\lim_{R \to \infty} \int_{S_R}
\left( \delta^{h_0} \tilde{g}_t - d \trace^{h_0} \tilde{g}_t \right) 
\lrcorner \dv^{h_0} \\
&=
\frac{1}{4\pi L}
\lim_{R \to \infty} \int_{S_R}
\left( \delta^{h_0} (\varphi_t^{2} g_t) - d \trace^{h_0} (\varphi_t^{2} g_t) \right) 
\lrcorner \dv^{h_0} \\
&=
\frac{1}{4\pi L}
\lim_{R \to \infty} \int_{S_R}
\varphi_t^{2} \left( \delta^{h_0} g_t - d \trace^{h_0} g_t \right) 
\lrcorner \dv^{h_0}
-\frac{6}{4\pi L}
\lim_{R \to \infty} \int_{S_R}
\varphi_t d\varphi_t \lrcorner \dv^{h_0} \\
&=
\mu^D_{g_t}
-\frac{6}{4\pi L} 
\lim_{R \to \infty} \int_{S_R}
\varphi_t d\varphi_t \lrcorner \dv^{h_0} \\
&=
\mu^D_{g^{RS}}
-
\frac{6}{4\pi L} 
\lim_{R \to \infty} \int_{S_R}
\nu^{h_0}(\varphi_t) \, \dv^{h_0}, \\
\end{split} \]
where $\nu^{h_0}$ is the outward pointing normal to $S_R$ with respect to $h_0$.
From the equation for $\varphi_t$ we have
\[
0 = \int_{B_R} \left( -6\Delta^{g_t} \varphi_t + \scal^{g_t} \varphi_t \right) \dv^{g_t} 
= -6\int_{S_R} \nu^{g_t}(\varphi_t) \, \dv^{g_t}
+ \int_{B_R} \scal^{g_t} \varphi_t \, \dv^{g_t},
\]
so in the limit $R \to \infty$ we find
\[
6 \lim_{R \to \infty} \int_{S_R} \nu^{h_0}(\varphi_t) \, \dv^{h_0} 
= 6\lim_{R \to \infty} \int_{S_R} \nu^{g_t}(\varphi_t) \, \dv^{g_t} 
= \int_{M} \scal^{g_t} \varphi_t \, \dv^{g_t}.
\]
Together we get
\[
\mu^D_{\tilde{g}_t}
= \mu^D_{g^{RS}} - \frac{1}{4\pi L} \int_{M} \scal^{g_t} \varphi_t \, \dv^{g_t},
\]
so the mass of $(M,\tilde{g}_t)$ is strictly less than the mass of $(M,g^{RS})$.
\end{proof}

In \cite{MinerbeMassALF} the mass of asymptotically locally flat manifolds is defined.
A similar computation shows that the same conformal change produces mass decreasing perturbations of the Taub-Bolt metric.

\bibliographystyle{plain}
\bibliography{biblio}



\section*{Declarations}

\noindent\textbf{Data availibility statement.} No data associate for the submission.\\
\vspace{-1mm}

\noindent\textbf{Conflict of interest.} On behalf of all authors, the corresponding author states that there is no conflict of
interest.

\end{sloppypar}
\end{document}